%% file: FliegeTinZemkoho.tex
\documentclass[11pt, headinclude,footinclude, BCOR5mm]{scrartcl}
\input{structure.tex}
\usepackage{titlesec}
\titlelabel{\thetitle.\quad}
\usepackage[utf8]{inputenc}
\usepackage[english]{babel}
\usepackage{enumitem}
\usepackage{amsmath}
\usepackage{amssymb}
\usepackage{xcolor}
\usepackage{lipsum}
\usepackage{longtable}
\usepackage{amsthm}
\usepackage{mathtools}
\usepackage{amsfonts}
\usepackage{graphicx}
\usepackage{caption}
\usepackage{float}
\usepackage{dsfont}

\usepackage{algorithm}
\usepackage{algorithmic}
\makeatletter
\def\listofalgorithms{\@starttoc{loa}\listalgorithmname}
\def\l@algorithm{\@tocline{0}{3pt plus2pt}{0pt}{1.9em}{}}
\renewcommand{\ALG@name}{Algorithm}
\renewcommand{\listalgorithmname}{List of \ALG@name s}
\numberwithin{algorithm}{section}

\usepackage[margin=2cm]{geometry}
\theoremstyle{plain}
\numberwithin{equation}{section}
\newtheorem{dfn}{Definition}[section]
\newtheorem{lem}[dfn]{Lemma}
\newtheorem{thm}[dfn]{Theorem}

\newtheorem{alg}[dfn]{Algorithm}

\theoremstyle{definition}
\newtheorem{asm}[dfn]{Assumption}
\newtheorem{exm}[dfn]{Example}

\newcommand{\norm}[1]{\left\lVert#1\right\rVert}

\usepackage{etoolbox}
\patchcmd{\thebibliography}
  {\settowidth}
  {\setlength{\itemsep}{0pt plus 0.1pt}\settowidth}
  {}{}
\apptocmd{\thebibliography}
  {\small}
  {}{}

\makeatletter
\DeclareOldFontCommand{\rm}{\normalfont\rmfamily}{\mathrm}
\DeclareOldFontCommand{\sf}{\normalfont\sffamily}{\mathsf}
\DeclareOldFontCommand{\tt}{\normalfont\ttfamily}{\mathtt}
\DeclareOldFontCommand{\bf}{\normalfont\bfseries}{\mathbf}
\DeclareOldFontCommand{\it}{\normalfont\itshape}{\mathit}
\DeclareOldFontCommand{\sl}{\normalfont\slshape}{\@nomath\sl}
\DeclareOldFontCommand{\sc}{\normalfont\scshape}{\@nomath\sc}
\makeatother
\makeatletter
\def\@seccntformat#1{\@ifundefined{#1@cntformat}%
   {\csname the#1\endcsname\space}
   {\csname #1@cntformat\endcsname}}
\newcommand\section@cntformat{\thesection.\space}       
\newcommand\subsection@cntformat{\thesubsection.\space} 
\makeatother
\allowdisplaybreaks
\title{\normalfont\spacedallcaps{Gauss-Newton-type methods}\\\spacedallcaps{for bilevel optimization}} 
\author{J\"{o}rg Fliege$^{\natural}$, Andrey Tin$^{\dag}$, and Alain Zemkoho$^{\ddag}$\\
{\small{\emph{School of Mathematical Sciences, University of Southampton, SO17 1BJ Southampton, UK}}}}

\date{\today}
\begin{document}
\renewcommand{\sectionmark}[1]{\markright{\spacedlowsmallcaps{#1}}} 
\lehead{\mbox{\llap{\small\thepage\kern1em\color{halfgray} \vline}\color{halfgray}\hspace{0.5em}\rightmark\hfil}} 

\pagestyle{scrheadings}
\maketitle
\setcounter{tocdepth}{2}
\section*{Abstract}
\noindent
This article studies Gauss-Newton-type methods for over-determined systems to find solutions to bilevel programming problems.
To proceed, we use the lower-level value function reformulation of bilevel programs and consider necessary optimality conditions under appropriate assumptions. First under strict complementarity for upper- and lower-level feasibility constraints, we prove the convergence of a Gauss-Newton-type method in computing points satisfying these optimality conditions under additional tractable qualification conditions. Potential approaches to address the shortcomings of the method are then proposed, leading to alternatives such as the pseudo or smoothing Gauss-Newton-type methods for bilevel optimization. Our numerical experiments conducted on 124 examples from the recently released Bilevel Optimization LIBrary (BOLIB) compare the performance of our method under different scenarios and show that it is a tractable approach to solve bilevel optimization problems with continuous variables.

\let\thefootnote\relax\footnotetext{$\natural$ \textit{e-mail: \url{j.fliege@soton.ac.uk}.}}

\let\thefootnote\relax\footnotetext{$\dag$ \textit{e-mail: \url{a.tin@soton.ac.uk}.}}

\let\thefootnote\relax\footnotetext{$\ddag$ \textit{e-mail: \url{a.b.zemkoho@soton.ac.uk}.}}

\section{Introduction}\label{SecIntroduction}
We aim to solve the bilevel programming problem
\begin{equation}\label{initialbilev}
\underset{x, \, y}{\min \;}  F(x,y) \; \mbox{ s.t. } \; G(x, y) \leq 0, \;  y\in S(x):= \arg\underset{y}\min~\{f(x,y):\; g(x,y) \leq 0\},\\
\end{equation}
where $F:\mathds{R}^n \times \mathds{R}^m \rightarrow \mathds{R}$, $f:\mathds{R}^n \times \mathds{R}^m \rightarrow \mathds{R}$, $G:\mathds{R}^n \times \mathds{R}^m \rightarrow \mathds{R}^q$, and $g:\mathds{R}^n \times \mathds{R}^m \rightarrow \mathds{R}^p$.
As usual, we refer to $F$ (resp. $f$) as upper-level (resp. lower-level) objective function and $G$ (resp. $g$) stands for upper-level (resp. lower-level) constraint function. Solving problem \eqref{initialbilev} is very difficult because of the implicit nature of the lower-level optimal solution mapping $S : \mathds{R}^n \rightrightarrows \mathds{R}^m$ defined in \eqref{initialbilev}.

There are several ways to deal with the complex nature of problem \eqref{initialbilev}.
One popular technique is to replace the lower-level problem with its Karush-Kuhn-Tucker (KKT) conditions.
With this formulation, bilevel programming problems are strongly linked to MPECs (mathematical programs with equilibrium constraints), see, e.g., \cite{bilevelmpec10}, which are not necessarily easy to handle due in part to the extra variables representing the lower-level Lagrangian multipliers.
Interested readers are referred to \cite{Allendestill12, bilevelreform, Yanchong13}
and references therein, for results and methods based on this transformation. In this paper, we are going to use the lower-level value function reformulation (LLVF)
\begin{equation}\label{initialvalfuncform0}
\underset{x,\,y}\min~F(x,y) \;\mbox{ s.t. } \; G(x,y)\leq 0, \;\, g(x,y)\leq 0, \;\, f(x,y)\leq \varphi(x),
\end{equation}
where the optimal value function is defined by
\begin{equation}\label{LLVF}
  \varphi(x) := \inf~\left\{f(x,y)\left|~g(x,y)\leq 0\right.\right\},
\end{equation}
to transform problem \eqref{initialbilev} into a single-level optimization problem. As illustrated in \cite{newtonbilevel18}, this approach can provide tractable opportunities to develop second order algorithms for the bilevel optimization problem, as it does not involve first order derivatives for lower-level problem, as in the context of the KKT reformulation.

There are recent studies on solution methods for bilevel programs, based on the LLVF reformulation. 
For example, \cite{polyxeni141,polyxeni142,mitsos08,paulavicius17,wieseman13} develop global optimisation techniques for \eqref{initialbilev} based on \eqref{initialvalfuncform0}--\eqref{LLVF}. \cite{lin14,xu14,xu15} propose algorithms computing stationary points for \eqref{initialvalfuncform0}--\eqref{LLVF}, in the case where the upper-level and lower-level feasible sets do not depend on the lower-level and upper-level variable, respectively. \cite{newtonbilevel18} is the first paper to propose a Newton-type method for the LLVF reformulation for programs. Numerical results there show that the approach can be very successful. The system of equation build there is square while our method in this paper is based a non-square and overdetermined system of equations. Hence, the need to develop Gauss-Newton-type techniques to capture certain classes of bilevel optimization stationarity points.

One of the main problems in solving \eqref{initialvalfuncform0} is that its feasible points systematically fail many constraint qualifications (see, e.g., \cite{dempezemkoho1}). To deal with this issue, we will use the partial calmness condition \cite{optcondbil95}, to shift the value function constraint $f(x,y)\leq \varphi(x)$ to the upper-level objective function, as a penalty term with parameter $\lambda$.
The other major problem with the LLVF reformulation is that $\varphi$ is typically non-differentiable. This will be handled by using upper estimates of the subdifferential of the function; see, e.g., \cite{dempezemkoho1,newoptcond,bilevelreform,optcondbil95}.
Our Gauss-Newton-type scheme proposed in this paper is based on a relatively simple system of optimality conditions, which depends on $\lambda$.

To transform this optimality conditions into a system of equations, we substitute the corresponding complementarity conditions by the standard Fischer-Burmeister function \cite{fischer92}.  To deal with the non-differentiability of the Fischer-Burmeister function, we consider two approaches in this paper. The first one is to assume strict complementarity for the constraints involved in the upper- and lower-level feasible sets. As second option to avoid non-differentiability, we investigate a smoothing technique by adding a perturbation in the Fischer-Burmeister function. 

Another important aspect of the aforementioned system of equations is that it is \emph{overdetermined}.
Since overdetermined systems have non-square Jacobian, we cannot use a classical Newton-type method as in \cite{newtonbilevel18}.
Gauss-Newton and Newton-type methods with Moore-Penrose pseudo inverse are both introduced in Section \ref{SecNewton}.
It will be shown that these methods are well-defined for solving bilevel programs from the perspective of the LLVF reformulation \eqref{initialvalfuncform0}.
In particular, our framework ensuring that the Gauss-Newton method for bilevel optimization is well-defined, does not require any assumption on the lower-level objective function.
A strong link between these two methods is then discussed to show that they should perform very similarly for most problems.
Based on this relationship, we also expect the Newton method with pseudo inverse to be more robust, the evidence of which we will show in the numerical implementation in Section \ref{SecNumExperGaussNewt}.

We present results of extensive experiments for testing the methods and comparing with Matlab built-in function \emph{fsolve}.
In Section \ref{SecNumExperGaussNewt} the results are compared with known solutions of the problems to check if obtained stationary points are optimal solutions of the problems or not.
For 124 tested problems we obtain more than $80 \%$ of satisfying solutions in the sense of recovering known solutions with $<10\%$ error, or obtaining better ones by all methods with CPU time being less than half a second.
The number of recovered solutions as well as the performance profiles and feasibility check show that Gauss-Newton and Newton method with pseudo inverse outperform \emph{fsolve}.
It is worth mentioning that it is not typical to conduct such a number of experiments in the literature on testing solution methods for bilevel programming.
The conjecture of the similarity of the performance of the tested methods is verified numerically, also showing that Newton's method with pseudo inverse is indeed more robust than the classical Gauss-Newton method.
The technique for choosing the penalization parameter $\lambda$ is a heuristic that might depend on the structure of the problem.
%

\section{Optimality conditions and equation reformulation}\label{SecOptimalityCondition}

Let us start with some definitions required to state the main theorem of this section.
Define \emph{full convexity} of the lower-level as convexity of lower-level objective and all lower-level constraints with respect to all variables $(x,y)$.
Further on, a feasible point $(\bar{x},\bar{y}) \in \mathbb{R}^n \times \mathbb{R}^m$ is said to be \emph{lower-level regular} if there exists direction $d \in \mathds{R}^m$ such that
\begin{equation}\label{LMFCQ}
\nabla_y g_i(\bar{x},\bar{y})^T d < 0,\phantom{-} \text{for } i\in I_g(\bar{x},\bar{y}):=\{i: g_i (\bar{x}, \bar{y})=0\}.
\end{equation}
One can recognize that this is equivalent to \emph{MFCQ} holding for the lower-level constraints.
Similarly, for $(\bar{x},\bar{y}) \in \mathbb{R}^n \times \mathbb{R}^m$ satisfying the upper-level inequality constraints $G(\bar{x},\bar{y})$, $(\bar{x},\bar{y})$ is \emph{upper-level regular} if there exists a direction $d \in \mathds{R}^{n+m}$ such that
\begin{equation}\label{UMFCQ}
  \begin{array}{rl}
    \nabla G_j(\bar{x}, \bar{y})^T d < 0 & \text{for }  j\in I_G (\bar{x},\bar{y}):=\{j: G_j (\bar{x}, \bar{y})=0\},\\
    \nabla g_j(\bar{x}, \bar{y})^T d < 0 & \text{for }  j\in I_g (\bar{x},\bar{y}):=\{j: g_j (\bar{x}, \bar{y})=0\}.
  \end{array}
\end{equation}
Finally, to write the necessary optimality conditions for problem \eqref{initialvalfuncform0}, it is standard to use the following partial calmness concept \cite{optcondbil95}:
\begin{dfn}
Let $(\bar{x},\bar{y})$ be a local optimal solution of problem \eqref{initialvalfuncform0}. This problem is partially calm at $(\bar{x},\bar{y})$ if there exists $\lambda>0$ and a neighbourhood $U$ of $(\bar{x},\bar{y},0)$ such that
$$
F(x,y)-F(\bar{x},\bar{y})+\lambda |u| \geq 0,\;\, \forall (x,y, u)\in U:\; G(x,y)\leq 0, \; g(x,y)\leq 0, \; f(x,y)- \varphi(x)-u=0.
$$
\end{dfn}
According to \cite[Proposition 3.3]{optcondbil95}, problem \eqref{initialvalfuncform0}--\eqref{LLVF} being partially calm at a local optimal solution $(\bar{x},\bar{y})$ is equivalent to the existence of a parameter $\lambda>0$ such that $(\bar{x},\bar{y})$ is also a local optimal solution of problem
\begin{equation}\label{valuebilev2}
\underset{x,\,y}\min~F(x,y) + \lambda (f(x,y)-\varphi(x)) \;\mbox{ s.t. } \; G(x,y)\leq 0, \;\,  g(x,y)\leq 0.
\end{equation}
It is clear that this is a penalization of only the constraint $f(x,y)-\varphi(x)\leq 0$ with the penalty parameter $\lambda$.
Hence, problem \eqref{valuebilev2} is a usually labelled as a \emph{partial exact penalization} of problem \eqref{initialvalfuncform0}--\eqref{LLVF}.
With this reformulation it is now reasonable to assume standard constraint qualifications to derive optimality conditions.
Based on this, we have the following result, see, e.g., \cite{dempezemkoho1,newoptcond,bilevelreform,optcondbil95}, based on a particular estimate of the subdifferential of $\varphi$ \eqref{LLVF}.

\begin{thm}\label{BilevelOptTh}  Let $(\bar{x},\bar{y})$ be a local optimal solution to \eqref{initialvalfuncform0}--\eqref{LLVF}, where all function are assumed to be differentiable, $\varphi $ is finite around $\bar{x}$ and lower-level problem is fully convex.
 Further assume that the problem is partially calm at $(\bar{x},\bar{y})$, the lower-level regularity is satisfied at $(\bar{x},\bar{y})$ and upper-level regularity holds at $\bar{x}$. Then there exist $\lambda \geq0$, and Lagrange multipliers $u,v,w$ such that
\begin{align}
\nabla_x F(\bar{x}, \bar{y})+\nabla_x g(\bar{x}, \bar{y})^T (u - \lambda w)
+ \nabla_x G(\bar{x},\bar{y})^T v =0, \label{kktbilev12} \\
 \nabla_y F(\bar{x}, \bar{y}) +\nabla_y g(\bar{x}, \bar{y})^T (u-\lambda w) + \nabla_y G(\bar{x},\bar{y})^T v= 0, \label{kktbilev22} \\
\nabla_y f(\bar{x}, \bar{y}) + \nabla_y g(\bar{x}, \bar{y})^T w = 0, \label{kktbilev32} \\
u\geq 0, \;\; g(\bar{x}, \bar{y})\leq 0, \;\; u^T g(\bar{x}, \bar{y})=0, \label{kktbilev42} \\
v\geq 0, \;\; G(\bar{x}, \bar{y})\leq 0, \;\; v^T G(\bar{x},\bar{y})=0, \label{kktbilev52} \\
w\geq 0, \;\; g(\bar{x}, \bar{y})\leq 0, \;\; w^T g(\bar{x}, \bar{y})=0. \label{kktbilev62}
\end{align}
\end{thm}
Depending on the assumptions made, we can obtain optimality conditions different from the above.
The details of different stationarity concepts can be found in the latter references, as well as in \cite{zemkohothesis}.
Weaker assumptions will typically lead to more general conditions.
However, making stronger assumptions allows us to obtain systems that are easier to handle.
For instance, it is harder to deal with more general conditions introduced in Theorem 3.5 of \cite{dempezemkoho1} or Theorem 3.1 of \cite{newoptcond} because of the presence of the convex hull in the corresponding estimated of the subdifferential of $\varphi$ \cite{dempezemkoho1,newoptcond,bilevelreform,optcondbil95}.
The other advantage of \eqref{kktbilev12}-\eqref{kktbilev62} is that, unlike the system studied in \cite{newtonbilevel18}, these conditions do not require to introduce a new lower-level variable.

The above optimality conditions involve the presence of complementarity conditions \eqref{kktbilev42}-\eqref{kktbilev62}, which result from inequality constraints present in  \eqref{initialvalfuncform0}--\eqref{LLVF}. In order to reformulate the complementarity conditions in the form of a system of equations, we are going use the concept of  NCP-functions; see, e.g., \cite{sun_ncp}. The function
 $\phi:\mathds{R}^2 \rightarrow \mathds{R}$ is said to be a NCP-function if we have
$$
\phi(a,b)=0\;\; \iff \;\; a\geq 0, \;\; b \geq 0, \;\; ab=0.
$$
In this paper, we use $\phi(a,b) := \sqrt{a^2+b^2}-a-b$, known as the \emph{Fischer-Burmeister} function \cite{fischer92}. This leads to the reformulation of the optimality conditions \eqref{kktbilev12}--\eqref{kktbilev62} into the system of equations:
\begin{align}
\Upsilon^\lambda (z) := \left(\begin{array}{rr}  \nabla_x F(x, y)+\nabla_x g(x, y)^T (u - \lambda  w)
+ \nabla_x G(x,y)^T v \\
  \nabla_y F(x, y) +\nabla_y g(x, y)^T (u-\lambda w)+ \nabla_y  G(x,y)^T v\\
  \nabla_y f(x, y) + \nabla_y g(x, y)^T w \\
 \sqrt{u^2+g(x, y)^2} - u+g(x, y) \\
  \sqrt{v^2+G(x,y)^2 } - v+G(x,y) \\
  \sqrt{w^2+g(x, y)^2 } - w+g(x, y) \end{array} \right) =
0, \label{bilevncp21}
\end{align}
where we have $z:=(x, y, u, v, w)$ and
\begin{equation}\label{hunam}
\sqrt{u^2+g(x, y)^2} - u+g(x, y) :=
\left(\begin{array}{c}
\sqrt{u_1^2+g_1(x, y)^2} - u_1+g_1(x, y) \\
\vdots \\
\sqrt{u_p^2+g_p(x, y)^2} - u_p+g_p(x, y)
 \end{array} \right).
\end{equation}
$\sqrt{v^2+G(x,y)^2 } - v+G(x,y)$ and $ \sqrt{w^2+g(x, y)^2 } - w+g(x, y)$ are defined as in \eqref{hunam}.
The superscript $\lambda$ is used to emphasize the fact that this number is a parameter and not a variable for equation \eqref{bilevncp21}. One can easily check that this system made of $n+2m+p+q+ p$ real-valued equations and $n+m+p+q+p$ variables. Clearly, this means that \eqref{bilevncp21} is an over-determined system and the Jacobian of $\Upsilon^\lambda (z)$, when it exists, is a non-square matrix.

\section{Gauss-Newton-type methods under strict complementarity}\label{SecNewton}
To solve equation  \eqref{bilevncp21}, we use a Gauss-Newton-type method, as the system is over-determined. Hence,  it is necessary to compute the Jacobian of $\Upsilon^\lambda(z)$ \eqref{bilevncp21}. However, the function is not differentiable at any point where one of the pairs
$$
(u_i, \, g_i(x,y)), \; i=1, \ldots, p, \;\; (v_j, \, G_j(x,y)), \; j=1, \ldots, q, \;\; and \;\; (w_i, \, g_i(x,y)), \; i=1, \ldots, p
$$
vanishes. To avoid this situation, we assume throughout this section that the strict complementarity condition holds:
\begin{asm}\label{AssumStrictComp}
The strict complementarity condition holds at $(x, y, u, v, w)$ if
$(u_i, \, g_i(x,y))\neq 0$ and $(w_i, \, g_i(x,y))\neq 0$ for all $i=1, \ldots, p$ and $(v_j, \, G_j(x,y))\neq 0$ for all $j=1, \ldots, q$.
\end{asm}
Under this assumption, the Jacobian of $\Upsilon^\lambda$  is well-defined everywhere and hence, the \emph{Gauss-Newton step} to solve equation \eqref{bilevncp21} can be defined as
\begin{equation}\label{GnewtStep}
d^k =-(\nabla \Upsilon^\lambda (z^k)^T \nabla \Upsilon^\lambda(z^k))^{-1} \nabla \Upsilon^\lambda(z^k)^T \Upsilon^\lambda(z^k),
\end{equation}
provided that the involved inverse matrix exists; see, e.g., \cite{fletch,nocedalw}. This leads to the following algorithm tailored to equation  \eqref{bilevncp21}:

 \begin{alg}
Gauss-Newton Method for Bilevel Optimization
\label{algorithm 3}
\begin{algorithmic}
 \STATE \textbf{Step 0}: Choose $\lambda >0$, $\epsilon>0$, $K>0$, $z^0:=(x^0, y^0, u^0, v^0, w^0)$, and set $k:=0$.
 \STATE \textbf{Step 1}: If $\norm{\Upsilon^{\lambda}(z^k)}<\epsilon$ or $k\geq K$, then stop.
 \STATE \textbf{Step 2}: Calculate Jacobian $\nabla \Upsilon^\lambda(z^k)$ and compute the direction $d^k$ using \eqref{GnewtStep}.
 \STATE \textbf{Step 3}: Set $z^{k+1}:=z^k + d^k$, $k:=k+1$, and go to Step 1.
\end{algorithmic}
\end{alg}
 In Algorithm \ref{algorithm 3}, $\epsilon$ denotes the tolerance and $K$ is the maximum number of iterations. It is clear from \eqref{GnewtStep} that for the algorithm to be well-defined, the matrix $\nabla \Upsilon^\lambda (z)^T \nabla \Upsilon^\lambda (z)$ needs to be non-singular. In the next subsection, we provide tractable conditions ensuring that this is possible.

\subsection{Nonsingularity of $\nabla \Upsilon^\lambda (z)^T \nabla \Upsilon^\lambda (z)$ and Convergence}\label{SubSecInvertibility}

To proceed, first start by noting the following, for any matrix $A$ with more rows than columns,
the matrix $A^T A$ has full rank if and only if the columns of $A$ are linearly independent.
This result is important as full rank of $A^T A$  is equivalent to invertibility of $A^T A$.
As $\nabla \Upsilon^\lambda (z)$ is a $(n+2m+2p+q)\times (n+m+2p+q)$ matrix with $m$ more rows than columns, the linear independence of its columns ensures that  $\nabla \Upsilon^\lambda (z)^T \nabla \Upsilon^\lambda (z)$ is non-singular. It therefore suffices for us to provide conditions guarantying the linear independence of the columns of $\nabla \Upsilon^\lambda (z)$.

To present the Jacobian of the system \eqref{bilevncp21} in the compact form, let the \emph{upper-level} and \emph{lower-level Lagrangian} functions be defined by
$$
L^\lambda(z):= F(x,y) + g(x,y)^T (u-\lambda w) + G(x,y)^T v \; \mbox{ and } \; \mathcal{L} (z) := f(x,y)+g(x,y)^T w,
$$
respectively. As we need the derivatives of these functions in the sequel, we denote the Hessian matrices of $L^\lambda$ and $\mathcal{L}$, w.r.t. $(x,y)$, by
\begin{equation} \label{Lagranderiv}
\begin{array}{c}
\nabla^2 L^\lambda (z) := \left[\begin{array}{cc}
                   \nabla_{xx}^2 L^\lambda (z) &  \nabla_{yx}^2 L^\lambda (z)\\
   \nabla_{xy}^2 L^\lambda (z) &  \nabla_{yy}^2 L^\lambda (z)
                  \end{array}
  \right]         \quad \mbox{ and } \quad   \nabla (\nabla_y \mathcal L (z)) := \left[\begin{array}{lr}
                                                        \nabla_{xy}^2 \mathcal{L} (z) &  \nabla_{yy}^2 \mathcal{L} (z)
                                                    \end{array}\right]
\end{array}
\end{equation}
respectively.  Furthermore, letting
$
\nabla g(x,y)^T := \left[\begin{array}{c}
            \nabla_{x} g(x,y)^T \\
   \nabla_{y} g(x,y)^T
   \end{array}\right]
$
and
$
\nabla G(x,y)^T := \left[\begin{array}{c}
            \nabla_{x} G(x,y)^T \\
   \nabla_{y} G(x,y)^T
   \end{array}\right],
$
we can easily check that the Jacobian of $\Upsilon^\lambda (z)$ w.r.t. $z$ can be written as
\begin{equation}\label{compactJac}
\nabla \Upsilon^\lambda (z) = \left[
  \begin{array}{cccc}
    \nabla^2 L^\lambda  (z) & \nabla g(x,y)^T & \nabla G(x,y)^T & -\lambda \nabla g(x,y)^T \\
     \nabla (\nabla_y \mathcal{L} (z)) & O & O & \nabla_y g(x,y)^T \\
    \mathcal{T} \nabla g(x,y) & \Gamma & O & O \\
     \mathcal{A} \nabla G(x,y) & O &  \mathcal{B} & O \\
    \Theta \nabla g(x,y) & O & O & \mathcal{K} \\
  \end{array}
\right]
\end{equation}
with $\mathcal{T} :=diag\,\{\tau_1, \ldots,\tau_p\}$, $\Gamma := diag\,\{\gamma_1, \ldots,\gamma_p\}$, $\mathcal{A}: = diag\,\{\alpha_1, \ldots, \alpha_q\}$, $\mathcal{B} := diag\,\{\beta_1, \ldots, \beta_q\}$, $\Theta := diag\,\{\theta_1, \ldots, \theta_p\}$, and $\mathcal{K} := diag\,\{\kappa_1, \ldots, \kappa_p\}$, where the pair $(\tau_j, \gamma_j)$, $j:=1, \ldots p$ is defined by
\begin{equation}\label{deftau}
\tau_j:= \frac{g_j (x,y)}{\sqrt{u_j^2 + g_j(x,y)^2}}+1 \, \mbox{ and } \, \gamma_j:= \frac{u_j}{\sqrt{u_j^2 + g_j(x,y)^2}} - 1, \, \mbox{ for }\, j=1, \ldots p.
\end{equation}
The pairs $(\alpha_j, \beta_j)$, $j=1, \ldots, q$ and $(\theta_j, \kappa_j)$, $j=1, \ldots, p$ are defined similarly in terms of $(G_j(x,y), v_j)$, $j=1, \ldots,  q$ and $(g_j(x,y), w_j)$, $j=1, \ldots, p$, respectively.
Similarly to the lower-level (resp. upper-level) regularity condition in \eqref{LMFCQ} (resp. \eqref{UMFCQ}), we will need the lower-level (resp. upper-level) linear independence constraint qualification denoted by LLICQ (resp. ULICQ) and will be said to hold at a point $(\bar x, \bar y)$ if the family of gradients
\begin{equation}\label{LLICQ+ULICQ}
  \left\{\nabla_y g_i(\bar x, \bar y),\;\, i\in I_g(\bar x, \bar y)\right\}\;\; \left(\mbox{resp. }\;\left\{\nabla g_i (\bar{x},\bar{y}),\; i \in I_g(\bar x, \bar y), \;\;\nabla G_j(\bar{x},\bar{y}),\; j\in I_G(\bar x, \bar y)\right\} \right)
\end{equation}
is linearly independent.

\begin{thm}\label{GaussNewtLinIndep}
Let the point $\bar{z}=(\bar{x},\bar{y},\bar{u},\bar{v},\bar{w})$ satisfy the system \eqref{bilevncp21} for some $\lambda >0$. Suppose that Assumption \ref{AssumStrictComp} holds at $(\bar x, \bar y, \bar u, \bar v, \bar w)$, the LLICQ and ULICQ is satisfied at $(\bar x, \bar y)$, and  the matrix $\nabla^2 L^\lambda (\bar{z})$ is positive definite. Then, the columns of the Jacobian matrix $\nabla \Upsilon^\lambda (\bar{z})$ are linearly independent.
\end{thm}
\begin{proof} Consider an arbitrary vector $d :=(d^\top_1, d^\top_2, d^\top_3, d^\top_4)^T$ such that $\nabla \Upsilon^\lambda (\bar{z})d=0$ with the components $d_1\in \mathbb{R}^{n+m}$, $d_2\in \mathbb{R}^p$, $d_3\in \mathbb{R}^q$, and $d_4\in \mathbb{R}^p$. Then we have
\begin{align}
\nabla^2 L^\lambda (\bar z) d_1 + \nabla g(\bar{x}, \bar{y})^T d_2 +  \nabla G(\bar{x}, \bar{y})^T d_3   -\lambda \nabla g(\bar{x}, \bar{y})^T d_4 = 0, \label{indep1} \\
\mathcal{T} \nabla g(\bar{x}, \bar{y}) d_1 +  \Gamma d_2 = 0,  \label{indep2} \\
 \mathcal{A} \nabla G(\bar{x}, \bar{y}) d_1 + \mathcal{B} d_3 = 0,\label{indep3} \\
 \Theta \nabla g(\bar{x}, \bar{y}) d_1   +  \mathcal{K} d_4 = 0, \label{indep4} \\
\nabla (\nabla_y \mathcal{L} (\bar z)) d_1 +  \nabla_y g(\bar{x}, \bar{y})^T d_4 = 0. \label{indep5}
\end{align}
On the other hand, it obviously follows from \eqref{deftau} that
\begin{equation}\label{taugammprop}
(\tau_j-1)^2+(\gamma_j+1)^2=1 \;\; \mbox{ for }\;\, j=1, \ldots p.
\end{equation}
Hence, the indices of the pair $(\tau, \gamma)$ satisfying \eqref{taugammprop} can be partitioned into the sets
$$
 P_1:= \{j: \tau_j > 0, \;\, \gamma_j <0\}, \;\; P_2 :=\{j: \tau_j=0 \}, \;\; \mbox{ and } \;\; P_3 :=\{j: \gamma_j =0 \}.
$$
Similarly, define index sets $Q_1$, $Q_2$, and $Q_3$ for the pair $(\alpha, \beta)$ and $T_1$, $T_2$, and $T_3$ for $(\theta, \kappa)$. Next, consider the following componentwise description of \eqref{indep2}, \eqref{indep3}, and \eqref{indep4},
\begin{align}
\tau_j \nabla g_j(\bar{x}, \bar{y})^T d_1 +  \gamma_j d_{2_j} = 0 \phantom{-} \;\text{for } j=1,...,p, \label{indep21} \\
 \alpha_j \nabla G_j(\bar{x}, \bar{y})^T d_1 + \beta_j d_{3_j} = 0  \phantom{-} \;\text{for } j=1,...,q,  \label{indep31} \\
 \theta_j \nabla g_j(\bar{x}, \bar{y})^T d_1   +  \kappa_j d_{4_j} = 0  \phantom{-} \;\text{for } j=1,...,p. \label{indep41} 
\end{align}

For $j\in P_2$ equation \eqref{indep21} becomes $\gamma_j d_{2_j} = 0$. Additionally, it follows from \eqref{taugammprop}
 that  for $j\in P_2$, $\gamma_j \neq 0$. Hence $d_{2_j}=0$ for $j\in P_2$.
For $j\in P_3$,  \eqref{indep21} leads to $\tau_j \nabla g_j(\bar{x}, \bar{y})^T d_1 = 0$, which due to the property above translates into $\nabla g_j(\bar{x}, \bar{y})^T d_1 = 0$, as $\tau_j \neq 0$. Finally, for $j\in P_1$ equation (\ref{indep21}) takes the form $\nabla g_j(\bar{x}, \bar{y})^T d_1 = - \frac{\gamma_j}{\tau_j} d_{2_j}$, where by definition of $P_1$ we know that $- \frac{\gamma_j}{\tau_j}>0$.
Following the same logic, we respectively have from \eqref{indep31} and  \eqref{indep41} that
$$
\begin{array}{lll}
d_{3_j}=0\; \mbox{ for } \; j\in Q_2, &\nabla G_j(\bar{x}, \bar{y})^T d_1 = 0\; \mbox{ for } \; j\in Q_3, &\nabla G_j(\bar{x}, \bar{y})^T d_1 = - \frac{\beta_j}{\alpha_j} d_{3_j} \; \mbox{ for } \; j\in Q_1,\\
d_{4_j}=0\; \mbox{ for } \; j\in T_2, & \nabla g_j(\bar{x}, \bar{y})^T d_1 = 0\; \mbox{ for } \; j\in T_3, & \nabla g_j(\bar{x}, \bar{y})^T d_1 = - \frac{\kappa_j}{\theta_j} d_{4_j} \; \mbox{ for } \; j\in T_1,
\end{array}
$$
with $- \frac{\beta_j}{\alpha_j}>0$ for $j\in Q_1$ and $- \frac{\kappa_j}{\theta_j}>0$ for $j\in T_1$. Multiplying \eqref{indep1} by $d_1 ^T$,
\begin{equation}
d_1^T \nabla^2 L^\lambda (\bar{z}) d_1 + d_1^T \nabla g(\bar{x}, \bar{y})^T d_2 +  d_1^T \nabla G(\bar{x}, \bar{y})^T d_3   - \lambda d_1^T \nabla g(\bar{x}, \bar{y})^T d_4 = 0. \label{mainindep1}
\end{equation}
Considering the cases defined above, we know that for $j\in P_2$, $j\in Q_2$ and $j\in T_2$, the terms $d_{2_j}, d_{3_j}$ and $d_{4_j}$ disappear. For $j\in P_3$, $j\in Q_3$ and $j\in T_3$, the terms $\nabla g_j(\bar{x}, \bar{y})^T d_1$, $\nabla G_j(\bar{x}, \bar{y})^T d_1$ and $\nabla g_j(\bar{x}, \bar{y})^T d_1$ also vanish.
This leads to the equation \eqref{mainindep1} being simplified to
\begin{equation}
d_1^T \nabla^2 L^\lambda (\bar{z}) d_1 + \sum_{j\in P_1} \left(- \frac{\gamma_j}{\tau_j} \right) d_{2_j}^2 +  \sum_{j\in Q_1}  \left(- \frac{\beta_j}{\alpha_j}\right) d_{3_j}^2   - \lambda  \sum_{j\in T_1} \left(- \frac{\kappa_j}{\theta_j}\right) d_{4_j}^2  = 0. \label{mainindep2}
\end{equation}
One can easily check that thanks to Assumption \ref{AssumStrictComp}, the sets $P_1$, $Q_1$, and $T_1$ are empty. Hence, \eqref{mainindep2} reduces to
$
d_1^T \nabla^2 L^\lambda (\bar{z}) d_1  = 0,
$
which implies $d_1 = 0 $ under the positive definiteness of $\nabla^2 L^\lambda$.

We have shown that $d_{2_j}=0$, $d_{3_j}=0$ and $d_{4_j}=0$ for $j\in P_2$, $j\in Q_2$ and $j\in T_2$, and $d_{1_j} = 0 $ for all $j$.
Let us use these results to simplify equations \eqref{indep1} and \eqref{indep5}  as follows
\begin{align}
\sum_{j\in P_3} \nabla g_j(\bar{x}, \bar{y})^T d_{2_j} + \sum_{j\in Q_3} \nabla G_j(\bar{x}, \bar{y})^T d_{3_j}   - \lambda \sum_{j\in T_3}  \nabla g_j(\bar{x}, \bar{y})^T d_{4_j} = 0, \label{thirdindep1} \\
\sum_{j\in T_3} \nabla_y g_j (\bar{x}, \bar{y})^T d_{4_j} = 0. \label{thirdindep2}
\end{align}
Equation \eqref{thirdindep2} implies that $d_{4_j}=0$ for all $j\in T_3$, given that $T_3 \subseteq I_g(\bar x, \bar y)$ and the LLICQ holds at $(\bar{x}, \bar{y})$. Then \eqref{thirdindep1} becomes
\begin{equation}
\sum_{j\in P_3} \nabla g_j(\bar x, \bar y)^T d_{2_j} + \sum_{j\in Q_3} \nabla G_j(\bar x, \bar y)^T d_{3_j}=0, \nonumber
\end{equation}
which implies $d_{2_j}=0$ and $d_{3_j}=0$ for $j\in P_3$ and $j\in Q_3$, given that the ULICQ holds at $(\bar x, \bar y)$.
This completes the proof as we have shown that $\nabla \Upsilon^\lambda(\bar{z}) d = 0 $ only if $d=0$.
\end{proof}

\begin{exm}\label{ExampLinIndepAssum}
We consider an instance of problem \eqref{initialbilev} taken from the BOLIB Library \cite{bolib17} with
$$
\begin{array}{lll}
 \begin{array}{l}
 F(x, y) := (x-3)^2 + (y-2)^2,\\
 f(x, y) := (y-5)^2,
 \end{array}  &
 G(x, y) := \left(\begin{array}{c}  x-8\\
                                      -x
 \end{array} \right), & g(x, y) :=  \left(\begin{array}{c}  -2x+y-1\\
x-2y-2 \\
x+2y-14 \end{array} \right).
\end{array}
$$
The point
$\bar{z}=(\bar{x}, \bar{y}, \bar{u}_1, \bar{u}_2, \bar{u}_3, \bar{v}_1, \bar{v}_2, \bar{w}_1, \bar{w}_2, \bar{w}_3)=(1,3,4\lambda -2,0,0,0,0,4,0,0)$ satisfies equation \eqref{bilevncp21} for any $\lambda > 1/2$.
Obviously, strict complementarity holds at this point, for $\lambda>1/2$.
and the family of vectors $\{\nabla g_j (\bar{x},\bar{y}), j \in I_g(x,y), \nabla G_j(\bar{x},\bar{y}), j\in I_G(x,y)\}$ is linearly independent, as $I_g(x,y)=\{1\}$,  $I_G(x,y)=\emptyset$. It is easy to see that ULICQ holds as $\nabla g_1(\bar{x}, \bar{y})^T = (-2, 1)^T \neq 0$, and LLICQ holds as $\nabla_y g_1(\bar{x}, \bar{y})^T = 1 \neq 0$. Finally, we obviously have that $\nabla^2 L^\lambda (\bar{z})=2e$, where $e$ is the identity matrix of $\mathbb{R}^{2\times 2}$, is positive definite. In conclusion, this example shows that all assumptions of Theorem \ref{GaussNewtLinIndep} can hold for a bilevel program and therefore the Gauss-Newton method in \eqref{algorithm 3} is well-defined.
\end{exm}

Based on the result above, we can now state the convergence theorem for our Gauss-Newton Algorithm \ref{algorithm 3}.
To proceed, first note that by implementing the Gauss-Newton method to solve \eqref{bilevncp21}, leads to a solution to the least-square problem
\begin{equation}\label{Merit function}
\underset{z}\min~\Phi^\lambda(z):= \sum_{i=1}^{N+m} \Upsilon^\lambda_i(z)^2,
\end{equation}
where we define $N:=n+m+2p+q$.
The direction of the Newton method for problem \eqref{Merit function} can be written as
$$
d^k := - (\nabla\Upsilon^\lambda(z_k)^T \nabla\Upsilon^\lambda(z_k) + T(z_k))^{-1} \nabla\Upsilon^\lambda(z_k) \Upsilon(z_k),
$$
where $T(z_k) := \sum_{i=1}^N \Upsilon^\lambda_i (z_k) \nabla^2 \Upsilon^\lambda_i(z_k)$ is the term that is omitted in the Gauss-Newton direction \eqref{GnewtStep}.
It is well known that the Gauss-Newton method converges with the same rate as Newton method if the term $T(\bar{z})$ is small enough in comparison with the term $\nabla \Upsilon^\lambda(\bar{z})^T \nabla \Upsilon^\lambda(\bar{z})$; see, e.g., \cite{dennisbook96, nocedalw, SunYuan06}. This is the basis of the following convergence result of Algorithm \ref{algorithm 3}.

\begin{thm}  \label{convtheorem}
Let the assumptions in Theorem \ref{GaussNewtLinIndep} hold and suppose that $\bar{z}$ is a local optimal solution of problem \eqref{Merit function} and $\{z^k \}$ be a sequence of points generated by Algorithm \ref{algorithm 3} converging to $\bar{z}$.
Furthermore, assuming that $\nabla^2 L^\lambda$ and $\nabla (\nabla_y)\mathcal{L}$ are well-defined and Lipschitz continuous in the neighbourhood of $\bar{z}$,
$$
\norm{z^{k+1}-\bar{z} } \leq \norm{(\nabla \Upsilon^\lambda(\bar{z})^T \nabla \Upsilon^\lambda(\bar{z}))^{-1}} \norm{T(\bar{z})} \norm{z^k - \bar{z}} + O(\norm{z_k - \bar{z}}^2).
$$
\end{thm}
\begin{proof}
Start by recalling that under the assumptions of Theorem \ref{GaussNewtLinIndep}, the matrix $\nabla \Upsilon^\lambda(\bar{z})$ is of full column rank. Hence, it is positive definite. Furthermore, under the strict complementarity condition (made in Theorem \ref{GaussNewtLinIndep}) and well-definiteness of derivatives of $\nabla^2 L^\lambda$ and $\nabla (\nabla_y \mathcal{L}(z))$ in the neighbourhood of $\bar{z}$, all components of $z \mapsto \Upsilon^\lambda(z)$ and $z \mapsto \nabla \Upsilon^\lambda(z)$ are differentiable near $\bar{z}$.
Hence the term $T(z)$ is well-defined near $\bar{z}$.
Further on, Lipschitz continuity of derivatives of $\nabla^2 L^\lambda (z)$ and $\nabla (\nabla_y \mathcal{L}(z))$ implies Lipschitz continuity of $\nabla\Upsilon^\lambda(z)^T \nabla\Upsilon^\lambda(z)$ and of $T(z)$.
Hence the term $\big(\nabla\Upsilon^\lambda(z)^T \nabla\Upsilon^\lambda(z)+T(z) \big)$ is Lipschitz continuous in the neighbourhood of $\bar{z}$.
Finally, since $\nabla \Upsilon^\lambda(z)$ is differentiable near $\bar{z}$, the product $\nabla\Upsilon^\lambda(z)^T \nabla\Upsilon^\lambda(z)$ is also differentiable near $\bar{z}$.
We also know that $\nabla\Upsilon^\lambda(z)^T \nabla\Upsilon^\lambda(z)$ is non-singular in the neighbourhood of $\bar{z}$ under assumptions of Theorem \ref{GaussNewtLinIndep}.
By the Inverse Function Theorem, differentiability and non-singularity of $\nabla\Upsilon^\lambda(z)^T \nabla\Upsilon^\lambda(z)$ is sufficient to state that $\nabla\Upsilon^\lambda(z)^T \nabla\Upsilon^\lambda(z)$ is a diffeomorphism, and hence has a smooth and differentiable inverse $(\nabla\Upsilon^\lambda(z)^T \nabla\Upsilon^\lambda(z))^{-1}$ in the neighbourhood of $\bar{z}$.
 Smoothness and differentiability of $(\nabla\Upsilon^\lambda(z)^T \nabla\Upsilon^\lambda(z))^{-1}$ imply that $(\nabla\Upsilon^\lambda(z)^T \nabla\Upsilon^\lambda(z))^{-1}$ is Lipschitz continuous in the neighbourhood of $\bar{z}$.
 This leads to the result by applying \cite[Theorem 7.2.2]{SunYuan06}.
\end{proof}

With Theorem \ref{convtheorem} it is easy to see that Gauss-Newton method converges quadratically if $T(\bar{z})=0$
and Q-Linearly if $T(\bar{z})$ is small relative to $\nabla \Upsilon^\lambda(\bar{z})^T \nabla \Upsilon^\lambda(\bar{z})$.
Such properties could be satisfied for small residuals problems and for the problems that are not too nonlinear.
For the small residuals problems we have that the components  $\Upsilon^\lambda_i(\bar{z})$ are small for all $i$, which makes the term $T(\bar{z})$ small.
For the problems with not too much nonlinearity the components $\nabla^2 \Upsilon^\lambda_i(\bar{z})$ are small for all $i$, which also results in small $T(\bar{z})$.
If it turns out that we can obtain an exact solution $\Upsilon^\lambda(\bar{z})=0$, then $T(\bar{z})=0$, and we have quadratic convergence. It is worth noting that in general we cannot always have $\Upsilon^\lambda_i(\bar{z})=0$ for all $i$ as the system is overdetermined, but minimizing the sum of the squares of $\Upsilon^\lambda_i(\bar{z})$ we obtain a solution point $\bar{z}$, at which $\sum_{i=1}^{N+m} (\Upsilon^\lambda_i(z))^2$ is as small as possible.
If the problem has small residuals, then small $\Upsilon_i^\lambda (\bar{z})$ are naturally obtained by implementing Algorithm \ref{algorithm 3} as this is designed to minimize $\sum_{i=1}^{N+m} \Upsilon^\lambda_i(z)^2$.
In terms of having small components $\nabla^2 \Upsilon^\lambda_i(\bar{z})$ we observe that $\nabla^2 \Upsilon^\lambda(\bar{z})$ will involve many third derivatives of $F(x,y), G(x,y), f(x,y)$ and $g(x,y)$.
Hence, if the original problem is not too nonlinear, the Hessian of the system \eqref{bilevncp21} should be small.
As a result if there exists reasonable solution with $\Upsilon_i^\lambda(z) \approx0$ for all $i$ or if the original problem \eqref{initialbilev} is not too nonlinear, then the Gauss-Newton for Bilevel Programming converges Q-linearly.

The first drawback of Algorithm \ref{algorithm 3} is the requirement of strict complementarity in Assumption \ref{AssumStrictComp}, to help ensure the differentiability of the function $\Upsilon^\lambda$.
Assumption \ref{AssumStrictComp} is very strong and in the context of the pool of test problems used for our numerical experiments in Section \ref{SecNumExperGaussNewt}, for example, it did not hold at the last iteration for at least one value of $\lambda$ for the total of 54 out of the 124 problems.
If one wants to avoid the \emph{strict complementarity} assumption, one option is to use \emph{smoothing technique} for Fischer-Burmeister function; this will be discussed in Section \ref{GnewtSmoothed}.
Before we move to this, it is worth mentioning that a second issue faced by our Algorithm \ref{algorithm 3} is the requirement that the matrix $\nabla \Upsilon^\lambda (z) ^T \nabla \Upsilon^\lambda (z)$ be nonsingular at each iteration. To deal with this, one option is a Newton step, where the generalized inverse of  $\nabla \Upsilon^\lambda (z) ^T \nabla \Upsilon^\lambda (z)$, which always exists, is calculated. This is briefly discussed in the next subsection.

\subsection{Newton method with Moore-Penrose pseudo inverse}\label{SubSecPseudoNewton}
Indeed, one of the most challenging aspects of the Gauss-Newton step in Algorithm \ref{algorithm 3} is the computation of the inverse of the matrix $\nabla \Upsilon^\lambda (z_k)^T \nabla \Upsilon^\lambda (z_k)$, as this quantity might not exist at some iterations.
To deal with situations where the inverse of the matrix does not exist, various concepts of generalized inverse have been used in the context of the Newton method; see, e.g., \cite{pan91} for related details. Although we do not directly compute $\left(\nabla \Upsilon^\lambda (z_k)^T \nabla \Upsilon^\lambda (z_k)\right)^{-1}$ in our implementation of Algorithm \ref{algorithm 3} in Section \ref{SecNumExperGaussNewt}, we would like to compare the \emph{pure} Gauss-Newton-type method presented in the previous subsection and the Newton method with Moore-Penrose pseudo inverse. Hence, we present the later approach here and its relationship to Algorithm \ref{algorithm 3}.


For an arbitrary matrix $A \in \mathbf{R}^{m \times n}$, its \emph{Moore-Penrose pseudo inverse} (see, e.g., \cite{golub96}) is defined by
$$A^+ := V \Sigma^+ U^\top,  $$
where $V \Sigma^+ U^\top$ represents a singular value decomposition of $A$, where
$\Sigma^+$ corresponds to the pseudo-inverse of $\Sigma$ that can be given by
\begin{equation*}\label{Sigpinv}
\Sigma^+= \mbox{diag }\left(\frac{1}{\sigma_1}, \,\frac{1}{\sigma_2},\, \ldots,\, \frac{1}{\sigma_r}, \, 0, \, \ldots, \, 0\right) \;\; \mbox{ with }  \;\; r=\mbox{rank }(A).
\end{equation*}
if $A$ has full column rank, we have an additional property that
$$A^+ :=                 \left(A^T A\right)^{-1} A^\top .$$

Based on this definition, an iteration of the Newton method with pseudo inverse for system \eqref{bilevncp21} can be then stated as
\begin{equation}
z^{k+1}=z^k - \nabla \Upsilon(z^k)^{+} \Upsilon (z^k). \label{Newtpseudo}
\end{equation}
We are now going to refer to \eqref{Newtpseudo} as iteration of the \emph{Pseudo-Newton method}.
The Pseudo-Newton method for bilevel programming can be defined in the same fashion as Algorithm \ref{algorithm 3} with the difference that direction would be given by
$d^k =- \nabla \Upsilon^\lambda(z^k)^+ \Upsilon^\lambda(z^k)$.
Clearly, the Pseudo-Newton method is always well-defined, unlike the Gauss-Newton method, and hence, it will produce some result in the case when the Gauss-Newton method diverges  \cite{fletcher68}.
Based on this general behaviour and interplay between the two approaches, we will be comparing them in the numerical section. For details on the convergence of Newton-type methods with pseudo-inverse, the interested reader in referred to \cite{gatilov14}.

\section{Smoothing Gauss-Newton method}\label{GnewtSmoothed}

In this section, we relax the strict complementarity assumption, considering the fact that it often fails for many problems as illustrated in the previous section. However, to ensure the smoothness of the function $\Upsilon^\lambda$ \eqref{bilevncp21}, the Fischer-Burmeister function is replaced with the \emph{smoothing Fischer-Burmeister function} (see \cite{kanzow1996}) defined by 
\begin{equation} \label{FischerBurmsmooth}
\phi^\mu_{g_j} (x, y, u) :=\sqrt{u^2_j+g_j(x, y)^2 + 2\mu} - u_j+g_j(x, y) \;\;\mbox{ for }\;\; j=1, \ldots, p,
\end{equation}
where the perturbation parameter $\mu >0$ helps to guaranty its differentiability at points $(x, y, u)$ satisfying $u_j = g_j(x,y)=0$. It is well-known (see latter reference) that
\begin{equation}\label{equivalent-ce}
\phi^\mu_{g_j} (x, y, u) = 0 \;\; \Longleftrightarrow \;\; \left[ u_j>0,\; -g_j (x,y) >0, \; -u_j g_j(x,y) = \mu\right].
\end{equation}
The smoothing system of optimality conditions becomes
\begin{align}
\Upsilon^{\lambda}_\mu (z) :=  \left(\begin{array}{rr}  \nabla_x F(x, y)+\nabla_x g(x, y)^T (u - \lambda  w)
+ \nabla_x G(x,y)^T v \\
  \nabla_y F(x, y) +\nabla_y g(x, y)^T (u-\lambda w)+ \nabla_y  G(x,y)^T v\\
  \nabla_y f(x, y) + \nabla_y g(x, y)^T w \\
 \sqrt{u^2+g(x, y)^2 + 2\mu} - u+g(x, y) \\
  \sqrt{v^2+G(x,y)^2 + 2\mu } - v+G(x,y) \\
  \sqrt{w^2+g(x, y)^2  + 2\mu} - w+g(x, y) \end{array} \right) =
0, \label{bilevncp2222}
\end{align}
following the convention in \eqref{hunam}, where $\mu$ is a vector of appropriate dimensions with sufficiently small positive elements.
Under the assumption that all the functions involved in problem \eqref{initialbilev} are continuously differentiable, $\Upsilon^{\lambda}_\mu$ is also a continuously differentiable function for any $\lambda > 0$ and $\mu>0$. Additionally, we can easily check that
$$
\|\Upsilon^{\lambda}_{\mu} (z) - \Upsilon^{\lambda} (z)\| \longrightarrow 0 \;\; \mbox{ as } \;\; \mu \downarrow 0.
$$
Following the smoothing scheme discussed, for example, in \cite{QiSunAsurvey}, our aim is to consider a sequence $\{\mu_k\}$ decreasing to $0$ such that equation \eqref{bilevncp21} is approximately solved:
$$
\Upsilon^{\lambda}_{\mu^k} (z)=0, \;\;\; k = 0, 1, \ldots
$$
for a fixed value of $\lambda >0$. Hence, we consider the following algorithm for system \eqref{bilevncp2222}:
\begin{alg}
Smoothing Gauss-Newton Method for Bilevel Optimization
\label{algorithm 3s}
\begin{algorithmic}
 \STATE \textbf{Step 0}: Choose $\lambda >0$, $\mu_0 \in (0,1)$, $z^0:=(x^0,y^0,u^0,v^0,w^0)$, $\epsilon>0$, $K>0$, set $k:=0$.
 \STATE \textbf{Step 1}: If $\norm{\Upsilon^{\lambda}(z^k)}<\epsilon$ or $k\geq K$, then stop.
 \STATE \textbf{Step 2}: Calculate Jacobian $\nabla \Upsilon_{\mu_k}^\lambda(z^k)$ and find the direction
 \begin{equation}\label{Gauss-Newton-Smoothed}
d^k =-(\nabla \Upsilon_{\mu_k}^\lambda (z^k)^T \nabla \Upsilon_{\mu_k}^\lambda(z^k))^{-1} \nabla \Upsilon_{\mu_k}^\lambda(z^k)^T \Upsilon^\lambda(z^k).
 \end{equation}
 \STATE \textbf{Step 3}: Calculate $z^{k+1}=z^k + d^k$.
  \STATE \textbf{Step 4}: Update $\mu_{k+1} = \mu_k^{k+1}$.
  \STATE \textbf{Step 5}: Set $k:=k+1$ and go to Step 1.
\end{algorithmic}
\end{alg}
To implement Algorithm \ref{algorithm 3s} numerically, we compute the direction by solving
 \begin{equation}\label{Gauss-Newtonsdireq}
\nabla \Upsilon_{\mu_k}^\lambda (z^k)^T \nabla \Upsilon_{\mu_k}^\lambda(z^k) d^k =-\nabla \Upsilon_{\mu_k}^\lambda(z^k)^T \Upsilon^\lambda(z^k).
\end{equation}
The Pseudo-Newton algorithm for the smoothed optimality conditions \eqref{bilevncp2222} will be the same as Algorithm \ref{algorithm 3s} apart from \textbf{Step 2}, where the corresponding direction is given by
$$
d^k =-\nabla \Upsilon_{\mu_k}^\lambda (z^k)^+ \Upsilon^\lambda(z^k).
$$

Anther way to deal with the non-differentiability of the Fischer-Burmeister NCP-function is to introduce a generalized \emph{generalized Jacobian} concept for the system \eqref{bilevncp21}.
A semismooth Newton-type method for bilevel optimization following this type of approach is developed in \cite{newtonbilevel18}. But we are not taking this approach here.

Similarly to \eqref{deftau}--\eqref{compactJac}, we introduce the matrices
 $\mathcal{T}^{\mu}$, $\Gamma^{\mu}$, $\mathcal{A}^{\mu}$, $\mathcal{B}^{\mu}$, $\Theta^{\mu}$, and $\mathcal{K}^{\mu}$, where for instance, the pair
 $\left(\mathcal{T}^\mu, \Gamma^\mu\right)$ is defined by $\mathcal{T}^\mu := diag~\{\tau_1^\mu,..,\tau_p^\mu\}$ and $\Gamma^\mu := diag~\{\gamma_1^\mu,..,\gamma_p^\mu\}$ with
\begin{equation}\label{deftau-mu}
\tau^{\mu}_j:= \frac{g_j (x,y)}{\sqrt{u_j^2 + g_j(x,y)^2 + 2\mu}}+1 \; \mbox{ and } \; \gamma^{\mu}_j:= \frac{u_j}{\sqrt{u_j^2 + g_j(x,y)^2 +2\mu}} - 1, \;\, j = 1, \ldots p.
\end{equation}
With this notation, we can easily check that for $\lambda >0$ and $\mu>0$, the Jacobian of  $\Upsilon_\mu^\lambda$ is
\begin{equation} \label{compactJsmooth}
\nabla \Upsilon^\lambda_\mu (z) = \left[
  \begin{array}{cccc}
    \nabla^2 L^\lambda (z)  & \nabla g(x,y)^T & \nabla G(x,y)^T & -\lambda \nabla g(x,y)^T \\
     \nabla (\nabla_y\mathcal{L} (z)) & O & O & \nabla_y g(x,y)^T \\
    \mathcal{T}^\mu \nabla g(x,y) & \Gamma^\mu & O & O \\
     \mathcal{A}^\mu \nabla G(x,y) & O &  \mathcal{B}^\mu & O \\
    \Theta^\mu \nabla g(x,y) & O & O & \mathcal{K}^\mu \\
  \end{array}
\right]
\end{equation}

The fundamental difference between the framework here and the one in the previous section is that for the pair $(\tau_j^\mu, \gamma_j^\mu)$, $j=1, \ldots, p$, for instance, we have the strict inequalities
$$
(\tau_j^\mu-1)^2+(\gamma_j^\mu+1)^2 < 1, \;\; j=1, \ldots p
$$
instead of equalities in the context of $(\tau_j, \gamma_j)$, $j=1, \ldots, p$ \eqref{deftau}. The next lemma illustrates a further difference between the new coefficients in this section and the ones in  \eqref{deftau}.

\begin{lem}\label{coefsign}
For a point $z:=(x, y, u, v, w)$ and $\mu>0$ such that $\Upsilon^\lambda_\mu (z)=0$, it holds that
\[
\begin{array}{lll}
 \tau_j^\mu >0, & \gamma_j^\mu<0, & j=1, \ldots, p, \\
\alpha_j^\mu >0, & \beta_j^\mu<0, & j=1, \ldots, q, \\
\theta_j^\mu >0, & \kappa_j^\mu<0, & j=1, \ldots, p.
\end{array}
\]
\end{lem}
\begin{proof}
We prove that  $\tau_j^\mu >0$ and  $\gamma_j^\mu<0$ for $j=1, \ldots, p$; the other cases can be done similarly. For $j=1, \ldots, p$, it follows from  \eqref{equivalent-ce} that
 $g_j(x, y) = -\frac{\mu}{u_j}$. Hence, we can rewrite $\tau_j^\mu$ and $\gamma_j^\mu$ as
\begin{equation} \label{taunew}
\tau_j^\mu = 1-\frac{\mu}{u_j \sqrt{u_j^2 +\frac{\mu^2}{u_j^2}+ 2\mu}} \; \mbox{ and } \;
\gamma_j^\mu = \frac{u_j}{\sqrt{u_j^2 +\frac{\mu^2}{u_j^2}+ 2\mu}} - 1,
\end{equation}
respectively. Next, we consider the following three scenarios:\\

\textbf{Case 1} Suppose that $u_j = \mu$. Substituting this value into \eqref{taunew}, we get
\begin{align*}
\tau_j^\mu = 
1 - \frac{1}{\mu+1} > 0 \; \mbox{ and }\;
\gamma_j^\mu = 
 \frac{\mu}{\mu+1} - 1 < 0 \; \mbox{ as } \; \mu >0.
\end{align*}

\textbf{Case 2} Suppose that $u_j = \mu + \delta$ for some $\delta>0$ and substituting this in \eqref{taunew}  leads to
$$\tau_j^\mu = 1 - \frac{\mu}{(\mu+\delta)\sqrt{(\mu+\delta)^2+\frac{\mu^2}{(\mu+\delta)^2}+2\mu}} = 1 - \frac{1}{\sqrt{\frac{(\mu+\delta)^4}{\mu^2}+1+2\frac{(\mu+\delta)^2}{\mu}}} > 0,$$
$$\gamma_j^\mu = \frac{\mu+\delta}{\sqrt{(\mu+\delta)^2+\frac{\mu^2}{(\mu+\delta)^2}+2\mu}} - 1 = \frac{1}{\sqrt{1+\frac{\mu^2}{(\mu+\delta)^4}+2\frac{\mu}{(\mu+\delta)^2}}} - 1  < 0.$$

\textbf{Case 3} Finally, suppose that $u_j = \mu - \delta$ for some $\delta>0$. Then substituting this in \eqref{taunew},
$$\tau_j^\mu = 1 - \frac{\mu}{(\mu-\delta)\sqrt{(\mu-\delta)^2+\frac{\mu^2}{(\mu-\delta)^2}+2\mu}} = 1 - \frac{1}{\sqrt{\frac{(\mu-\delta)^4}{\mu^2}+1+2\frac{(\mu-\delta)^2}{\mu}}} > 0,$$
$$\gamma_j^\mu = \frac{\mu-\delta}{\sqrt{(\mu-\delta)^2+\frac{\mu^2}{(\mu-\delta)^2}+2\mu}} - 1 = \frac{1}{\sqrt{1+\frac{\mu^2}{(\mu-\delta)^4}+2\frac{\mu}{(\mu-\delta)^2}}} - 1  < 0.$$
Note that $u_j=\mu-\delta >0$ for in  the third case; this helps to ensure that $\mu-\delta = \sqrt{(\mu-\delta)^2}$. 
\end{proof}

Next, we use this lemma to provide a condition ensuring that the matrix $\nabla \Upsilon_\mu^\lambda(z)^T \nabla \Upsilon_\mu^\lambda (z)$ is nonsingular.  This will allow the smoothed Gauss-Newton step \eqref{Gauss-Newton-Smoothed} to be well-defined.
As in the previous section, it suffices to show that the columns of $\nabla \Upsilon_\mu^\lambda (\bar{z})$ are linearly independent.

\begin{thm}\label{GaussNewtLinIndeps}
For a point $\bar{z} :=(\bar{x},\bar{y},\bar{u},\bar{v},\bar{w})$ verifying \eqref{bilevncp2222} for some $\mu>0$ and $0 < \lambda < \frac{\kappa_j^\mu}{\theta_j^\mu} \frac{\tau_j^\mu}{\gamma_j^\mu}$, $j=1, \ldots, p$, suppose that $\nabla^2 L^\lambda (\bar{z})$ is positive definite. Then, the columns of the matrix $\nabla \Upsilon_\mu^\lambda (\bar{z})$ are linearly independent.
\end{thm}
\begin{proof}
Similarly to the proof of Theorem \ref{GaussNewtLinIndep}, $\nabla \Upsilon_\mu^\lambda (\bar{z})\left(d^\top_1,
                                                                                                   d^\top_2,
                                                                                                   d^\top_3,
                                                                                                   d^\top_4\right)^\top=0$ is equivalent to
\begin{align}
\nabla^2 L^\lambda (\bar{z}) d_1 + \nabla g(\bar{x}, \bar{y})^T d_2 +  \nabla G(\bar{x}, \bar{y})^T d_3   -\lambda \nabla g(\bar{x}, \bar{y})^T d_4 = 0, \label{undep1s1} \\
\tau_j^\mu \nabla g_j(\bar{x}, \bar{y})^T d_1 +  \gamma_j^\mu d_{2_j} = 0,  \label{undep2s1} \\
 \alpha_j^\mu \nabla G_j(\bar{x}, \bar{y})^T d_1 + \beta_j^\mu d_{3_j} = 0,\label{undep3s1} \\
 \theta_j^\mu \nabla g_j(\bar{x}, \bar{y})^T d_1   +  \kappa_j^\mu d_{4_j} = 0, \label{undep4s1} \\
\nabla (\nabla_y \mathcal{L}(\bar{z})) d_1 +  \nabla_y g(\bar{x}, \bar{y})^T d_4 = 0, \label{undep5s1}
\end{align}
where $j=1,...,p$ in \eqref{undep2s1} and \eqref{undep4s1}, while $j=1,...,q$ in \eqref{undep3s1}.
Thanks to  Lemma \ref{coefsign}, we can rewrite equations \eqref{undep2s1}, \eqref{undep3s1},  and \eqref{undep4s1} as
\begin{equation}\label{Yola1}
\nabla g_j(\bar{x}, \bar{y})^\top d_1 = - \frac{\gamma_j^\mu}{\tau_j^\mu} d_{2_j}, \;\, \nabla G_j(\bar{x}, \bar{y})^\top d_1 = - \frac{\beta_j^\mu}{\alpha_j^\mu} d_{3_j}, \;\, \mbox{ and } \;\; \nabla g_j(\bar{x}, \bar{y})^\top d_1 = - \frac{\kappa_j^\mu}{\theta_j^\mu} d_{4_j},
\end{equation}
respectively, with $- \frac{\gamma_j^\mu}{\tau_j^\mu}>0$, $-\frac{\beta_j^\mu}{\alpha_j^\mu} >0$, and $- \frac{\kappa_j^\mu}{\theta_j^\mu}>0$. Now, let us multiply \eqref{undep1s1} by $d^\top_1$:
\begin{equation}
d_1^T \nabla^2 L^\lambda (\bar{z}) d_1 + d_1^\top \nabla g(\bar{x}, \bar{y})^\top d_2 +  d_1^\top \nabla G(\bar{x}, \bar{y})^\top d_3   - \lambda d_1^\top \nabla g(\bar{x}, \bar{y})^\top d_4 = 0. \label{mainundep1s}
\end{equation}
Using the results above, the equation \eqref{mainundep1s} can be written as
\begin{equation}
d_1^T \nabla^2 L^\lambda (\bar{z}) d_1 + \sum_{j=1}^p \left(- \frac{\gamma_j^\mu}{\tau_j^\mu}\right) d_{2_j}^2 +  \sum_{j=1}^q  \left(- \frac{\beta_j^\mu}{\alpha_j^\mu}\right) d_{3_j}^2   - \lambda  \sum_{j=1}^p \left(- \frac{\kappa_j^\mu}{\theta_j^\mu}\right) d_{4_j}^2  = 0. \label{mainundep2s}
\end{equation}
Furthermore, it follows from the first and last items of \eqref{Yola1} that
\begin{equation}\label{d4jd2j}
d_{4_j} =  - \frac{\theta_j^\mu}{\kappa_j^\mu}\nabla g_j(\bar{x}, \bar{y})^\top d_1 =  \frac{\theta_j^\mu \gamma_j^\mu}{\kappa_j^\mu \tau_j^\mu} d_{2_j}.
\end{equation}
Substituting \eqref{d4jd2j} into \eqref{mainundep2s}
\begin{equation}
d^\top_1 \nabla^2 L^\lambda (\bar{z}) d_1 + \sum_{j=1}^p \left(- \frac{\gamma_j^\mu}{\tau_j^\mu}\right) d_{2_j}^2 +  \sum_{j=1}^q  \left(- \frac{\beta_j^\mu}{\alpha_j^\mu}\right) d_{3_j}^2   - \lambda \sum_{j=1}^p \left(- \frac{\kappa_j^\mu}{\theta_j^\mu}\right) \left(\frac{\theta_j^\mu \gamma_j^\mu}{\kappa_j^\mu \tau_j^\mu}\right)^2 d_{2_j}^2  = 0.
\end{equation}
Rearranging this equation, we get
\begin{equation}\label{smfinal}
d_1^T \nabla^2 L^\lambda (\bar{z}) d_1 +  \sum_{j=1}^q  \left(- \frac{\beta_j^\mu}{\alpha_j^\mu}\right) d_{3_j}^2   + \sum_{j=1}^p  \left(1 - \lambda  \frac{\theta_j^\mu}{\kappa_j^\mu} \frac{\gamma_j^\mu}{\tau_j^\mu} \right)
 \left(-\frac{\gamma_j^\mu}{\tau_j^\mu}\right) d_{2_j}^2 =0.
\end{equation}
Then with the result in Lemma \ref{coefsign}, under  the assumptions  that $ \nabla^2 L^\lambda (\bar{z})$ is positive definite and $\lambda < \frac{\kappa_j^\mu}{\theta_j^\mu} \frac{\tau_j^\mu}{\gamma_j^\mu}$ for $j=1, \ldots, p$,  equation \eqref{smfinal} is the sum of non-negative terms, which can only be a sum of zeros if all components of $d_1, d_2$ and $d_3$ are zeros.
Since all components of $d_2$ are zeros, we can look back to \eqref{d4jd2j} or \eqref{undep4s1}  to deduce that $d_{4_j} = 0$ for  $j=1, \ldots p$, completing the proof.
\end{proof}

It is important to note that the assumption that $\lambda < \frac{\kappa_j^\mu}{\theta_j^\mu} \frac{\tau_j^\mu}{\gamma_j^\mu}$ does not necessarily conflict with the requirement that $\lambda$ be strictly positive, as due to Lemma \ref{coefsign}, we have $\frac{\kappa_j^\mu}{\theta_j^\mu} \frac{\tau_j^\mu}{\gamma_j^\mu}>0$. In Subsection \ref{SecLamassum}, a numerical analysis of this condition is conducted. Next, provide an example of bilevel program, where the assumptions made in the Theorem \ref{GaussNewtLinIndeps} can be satisfied.

\begin{exm}\label{ExampIndepAssumSmooth}
We consider an instance of problem \eqref{initialbilev} taken from the BOLIB Library \cite{bolib17} with
$$
\begin{array}{lll}
 \begin{array}{l}
 F(x, y) := x^2 + (y_1+y_2)^2,\\
 f(x, y) := y_1,
 \end{array}  &
 G(x, y) := -x+0.5,
  & g(x, y) :=  \left(\begin{array}{c}  -x-y_1-y_2+1\\
-y  \end{array} \right).
\end{array}
$$
The point
$\bar{z}=(\bar{x}, \,\bar{y}_1, \, \bar{y}_2, \, \bar{u}_1, \, \bar{u}_2, \, \bar{u}_3, \, \bar{v}, \, \bar{w}_1,\, \bar{w}_2, \, \bar{w}_3)=(0.5, \, 0, \, 0.5, \, 1, \, \lambda, \, 0, \, 0, \, 0, \, 1, \, 0)$ satisfies equation \eqref{bilevncp21} for any $\lambda > 0$.
Strict complementarity does not hold at this point as $(\bar v , G(\bar x, \bar y) ) =(0,0)$ and  $(\bar w_1 , g_1(\bar x, \bar y) ) =(0,0)$.
We observe that $\nabla^2 L^\lambda (\bar{z})=2e$, where $e$ is the identity matrix of $\mathbb{R}^{3\times 3}$, is positive definite.
As for the conditions $\lambda < \frac{\kappa_j^\mu}{\theta_j^\mu} \frac{\tau_j^\mu}{\gamma_j^\mu}$, $j=1,2,3$, they hold for any value of $\lambda$ such that
\[
0<\lambda < \min~\left\{\frac{1}{1-1/(2  \mu + 1)^{1/2}}, \;\;   \frac{1 - 1/(2\mu + 1)^{1/2}}{1  - \lambda /(\lambda^2 + 2 \mu)^{1/2}}, \;\; 1\right\}
\]
This is automatically the case if, for example, we set $\mu=2\times10^{-2}$ and $\lambda=10^{-2}$. \qed
\end{exm}

There is at least one other way to show that $\nabla \Upsilon_\mu^\lambda(\bar z)^T \nabla \Upsilon_\mu^\lambda (\bar z)$ is nonsingular. The approach is based on the structure of the matrix, as it will be clear in the next result. To proceed, we need the following two assumptions.
\begin{asm}\label{rowassum}
Each row of the following matrix is a nonzero vector:
\[
\left[\nabla^2 L^\lambda (z)^T \quad \nabla (\nabla_y\mathcal{L} (z))^T   \quad \nabla g(x,  y)^T \quad \nabla G(x, y)^T\right].
\]
\end{asm}
\begin{asm}\label{diagassum}
For $\lambda>0$ and $\mu>0$, the diagonal elements of the matrix $\nabla \Upsilon_\mu^\lambda(z)^T \nabla \Upsilon_\mu^\lambda (z)$
dominate the other terms row-wise; i.e., 
$a_{ii} > \sum^{N}_{j=1, \; j\neq i} |a_{ij}|$ for $i=1, \ldots, N$, where $a_{ij}$ denotes the element in the cell $(i, j)$ of  $\nabla \Upsilon_\mu^\lambda(z)^T \nabla \Upsilon_\mu^\lambda (z)$ for $i=1, \ldots, N$ and $j=1, \ldots, N$.
\end{asm}

%

\begin{lem}\label{diagdom}
Let Assumption \ref{rowassum} hold at the point $z :=(x, y, u, v, w)$. Then for any $\lambda >0$ and $\mu >0$, the diagonal elements of the matrix $\nabla \Upsilon_\mu^\lambda(z)^T \nabla \Upsilon_\mu^\lambda (z)$ are strictly positive.
\end{lem}
\begin{proof}
Considering the Jacobian matrix in \eqref{compactJsmooth}, its transpose can be written as
\begin{equation*}
\nabla \Upsilon^\lambda_\mu (z)^T = \left[
  \begin{array}{ccccc}
    \nabla^2 L^\lambda (z)^T  & \nabla (\nabla_y\mathcal{L} (z))^T &   \nabla g(x,y)^T \mathcal{T}^{\mu T}  &   \nabla G(x,y)^T \mathcal{A}^{\mu T} &   \nabla g(x,y)^T \Theta^{\mu T} \\
    \nabla g(x,y) & O & \Gamma^{\mu T} & O & O \\
    \nabla G(x,y) & O & O & \mathcal{B}^{\mu T} & O \\
   -\lambda \nabla g(x,y) & \nabla_y g(x,y) & O & O & \mathcal{K}^{\mu T} \\
  \end{array}
\right].
\end{equation*}
Denote by $r_i$, $i=1, \ldots, 4$, respectively, the first, second, third, and fourth row-block of this matrix. 
Then the desired product can be represented as
\begin{equation*}
\nabla \Upsilon_\mu^\lambda(z)^T \nabla \Upsilon_\mu^\lambda (z)= \left[
  \begin{array}{ccccc}
  r_1 r_1^T & r_1 r_2 ^T & r_1 r_3^T & r_1 r_4^T \\
   r_2 r_1^T & r_2 r_2 ^T & r_2 r_3^T & r_2 r_4^T \\
    r_3 r_1^T & r_3 r_2 ^T & r_3 r_3^T & r_3 r_4^T \\
     r_4 r_1^T & r_4 r_2 ^T & r_4 r_3^T & r_4 r_4^T \\
   \end{array}
 \right].
\end{equation*}
Obviously, the diagonal elements of $\nabla \Upsilon_\mu^\lambda(z)^T \nabla \Upsilon_\mu^\lambda (z)$ are the diagonal elements of $r_1 r_1^T$, $r_2 r_2^T$, $r_3 r_3^T$, and $r_4 r_4^T$. We can check that for $j=1, \ldots, n+m$, a diagonal element of $r_1 r_1^T$ has the form
\begin{equation*}
  \begin{array}{lll}
 (r_1 r_1^T)_{jj} &=&  \sum_{k=1}^{n+m} \nabla_{j,k}^2 L^\lambda(z) ^T \nabla_{j,k}^2 L^\lambda (z)  + \sum_{k=1}^{m}  \nabla_j (\nabla_{y_k} \mathcal{L} (z))^T\nabla_j (\nabla_{y_k}\mathcal{L} (z))\\
                  & & +  \sum_{k=1}^{p} \nabla_j g_k(x,y)^T \nabla_j g_k(x,y)  (\tau_k^\mu)^2  +    \sum_{k=1}^{q}  \nabla G_k(x,y)^T  \nabla G_k(x,y) (\alpha_k^\mu)^2 \\
                  & & +  \sum_{k=1}^{p}  \nabla g_k(x,y)^T  \nabla g_k(x,y) (\theta_k^\mu)^2,
  \end{array}
\end{equation*}
where $\nabla_j$ stands for the $j^{th}$ element of $\nabla:=\left(\nabla_{x_1},...,\nabla_{x_n},\nabla_{y_1},...,\nabla_{y_m}\right)$ and $\nabla^2_{j,k}$ corresponds to an element in the $j^{th}$ row and $k^{th}$ column of
\begin{align*}
\nabla^2 : =  \left[
  \begin{array}{cccc} \nabla_{x_1 x_1} & \dots \nabla_{x_1 x_n} & \nabla_{x_1 y_1} & \dots\nabla_{x_1 y_m} \\
  \vdots & \ddots & \ddots & \vdots \\
\nabla_{x_n x_1} & \dots\nabla_{x_n x_n} & \nabla_{x_n y_1} & \dots\nabla_{x_n y_m} \\
\nabla_{y_1 x_1} & \dots\nabla_{y_1 x_n} & \nabla_{y_1 y_1} & \dots\nabla_{y_1 y_m} \\
\vdots & \ddots & \ddots & \vdots \\
\nabla_{y_m x_1} & \dots\nabla_{y_m x_n} & \nabla_{y_m y_1} & \dots\nabla_{y_m y_m}
\end{array}
\right].
\end{align*}
Combining Assumption \ref{rowassum} and Lemma \ref{coefsign}, it is clear that $(r_1 r_1^T)_{jj}> 0$ for $j=1, \ldots, n+m$.
Similarly, the diagonal elements of $r_2 r_2^T$, $r_3 r_3^T$, and $r_4 r_4^T$ can respectively be written as
\begin{equation*}
\begin{array}{lll}
(r_2 r_2^T)_{jj} = \nabla g_j (x,y) \nabla g_j(x,y)^T + (\gamma_j^\mu)^2 & \mbox{ for } & j=1,...,p,\\
(r_3 r_3^T)_{jj} = \nabla G_j (x,y) \nabla G_j(x,y)^T + (\beta_j^\mu)^2 & \mbox{ for } & j=1,...,q,\\
(r_4 r_4^T)_{jj} = \nabla g_j (x,y) \nabla g_j(x,y)^T + (\kappa_j^\mu)^2 & \mbox{ for } & j=1,...,p.
\end{array}
\end{equation*}
Thanks to Lemma \ref{coefsign}, it is also clear that these items are all strictly positive. 
\end{proof}

\begin{thm}\label{GaussNewtNonsing}
Let $z=(x, y, u, v, w)$ be a stationary point of the system \eqref{bilevncp2222} for some $\lambda >0$ and $\mu >0$. If Assumptions \ref{rowassum} and \ref{diagassum} are satisfied, then the matrix $\nabla \Upsilon_\mu^\lambda(z)^T \nabla \Upsilon_\mu^\lambda (z)$ is nonsingular.
\end{thm}
\begin{proof}
It is known that the matrix is positive definite if it is symmetric, its diagonal elements are strictly positive, and diagonal elements dominate elements of the matrix in the corresponding row.
This property is the consequence of the \emph{Gershgorin circle theorem}, which can be found in \cite[page 320]{golub96}.
As  $\nabla \Upsilon_\mu^\lambda(z)^T \nabla \Upsilon_\mu^\lambda (z)$ is symmetric, then combining Assumptions \ref{rowassum} and \ref{diagassum} to Lemma \ref{diagdom}, we have the result.
\end{proof}

Next, we provide an example where the assumptions required for Theorem \ref{GaussNewtNonsing} are satisfied.
\begin{exm}\label{ExampSmoothDempe}
We consider an instance of problem \eqref{initialbilev} taken from the BOLIB Library \cite{bolib17} with
$$
 \begin{array}{lll}
 F(x, y) := (x-1)^2 + y^2, &
  f(x, y) := x^2 y, &
   g(x, y) :=  y^2
  \end{array}
$$
and no upper-level constraint. For this problem the function $\Upsilon^{\lambda}$ \eqref{bilevncp21} can be written as
\[
\Upsilon^{\lambda} (z)= \left(2x-2, \;
  2y+2yu -2 \lambda y w, \;
  x^2 + 2 y w, \;
\sqrt{u^2+y^4} - u + y^2, \;
\sqrt{w^2+y^4} - w + y^2\right)^\top.
\]
The first item of note about this example is that the optimal solution $(\bar x, \bar y)=(1, 0)$ does not satisfy that the optimality conditions \eqref{kktbilev12}--\eqref{kktbilev62}, given that $\Upsilon^{\lambda} (\bar x, \bar y, \bar u, \bar w)\neq 0$ for any values of $\bar u$ and $\bar w$. However, Algorithm \ref{algorithm 3s} identifies the solution for $\lambda$ taking the values $0.6$, $0.7$ or $0.8$ with the smoothing parameter set to $\mu=10^{-11}$. Indeed, the convergence of the method seems to be justified as for this problem, we can easily check that for $\bar x =1$ and $\bar y=0$,
\[
\begin{array}{l}
\nabla^2 L^\lambda(\bar z) =
 2\mbox{diag}\left(1, \; 1 + \bar u - \lambda \bar{w}\right), \quad \nabla (\nabla_y \mathcal{L}(\bar z)) = ( 2 \;\; 2 \bar w), \quad  \nabla g(\bar x, \bar y) = (0 \;\; 0),\\[1ex]
 \Gamma^\mu =  \frac{\bar u}{\sqrt{\bar u^2 + 2 \mu}}- 1, \;\; \mbox{ and } \;\; \mathcal{K}^\mu = \frac{\bar w}{\sqrt{\bar w^2 + 2 \mu}}- 1
\end{array}
\]
and subsequently, we have the product
\begin{align*}
\nabla \Upsilon_\mu^\lambda(\bar z)^T \nabla \Upsilon_\mu^\lambda (\bar z) =
\left[
\begin{array}{cccc}
  8  &                        4 \bar w  & 0  &                               0 \\
 4 \bar w  &   4 \bar w^2 + (2 \bar u - 2 \lambda \bar w )^2 & 0  &                               0 \\
   0  &                          0  & \left(\frac{\bar u}{\sqrt{\bar{u}^2 + 2 \mu}} - 1\right)^2  &                               0 \\
   0  &                          0  & 0  & \left(\frac{\bar w}{\sqrt{\bar w^2 + 2 \mu}} - 1\right)^2
 \end{array}
 \right].
 \end{align*}
Hence, Assumption \ref{rowassum} is clearly satisfied and for Assumption \ref{diagassum} to hold, we need
\[
8 > 4 \bar w, \;\, 4 \bar w^2 + (2 \bar u - 2 \lambda \bar w)^2 >  4 \bar w,\;\, \left(\frac{\bar u}{\sqrt{\bar u^2 + 2 \mu}} - 1\right)^2 > 0, \;\, \left(\frac{\bar w}{\sqrt{\bar w^2 + 2 \mu}} - 1\right)^2 > 0.
\]
This holds for any $\mu>0$, $\lambda>0$, $\bar u > 0$, and $1 < \bar w < 2$. \qed
\end{exm}

We further note that the assumptions made in Theorem \ref{GaussNewtLinIndeps} hold for the problem in this example. Firstly, we observe that
$\nabla^2 L^\lambda(\bar z)$ is positive definite if $\lambda \bar w < \bar u + 1$. Subsequently, we can check that both assumptions of Theorem \ref{GaussNewtLinIndeps} are satisfied if
\[
\lambda < \min \left\{\frac{\bar u + 1}{\bar w},    \frac{\left(\frac{\bar w}{\sqrt{\bar w^2 + 2 \mu}} -1 \right)} { \left(\frac{\bar u}{\sqrt{\bar u^2 + 2 \mu}} -1 \right)} \right\} \;\;\mbox{ with }\;\; \bar w \neq 0.
\]
For instance, choosing $\bar u = \sqrt{8} \times 10^{-6} $, $\bar w = 10^{-6}$, $\mu = 4\times 10^{-12}$, and $\lambda<2.25$ gives the result.

To conclude this section, we would like to analyse the \emph{Jacobian consistency} of $\Upsilon^\lambda$. Recall that according to  \cite{globandsuperlin98}, the Jacobian consistency property will hold for $\Upsilon^\lambda$ if this mapping is Lipschitz continuous and there exists a constant $\epsilon>0$ such that for any $z \in \mathbb{R}^N$ and $\mu \in \mathbb{R}_+$, we have
\begin{equation}\label{Jacobian consistency}
\norm{ \Upsilon^\lambda_\mu(z)- \Upsilon^\lambda(z)} \leq \mu \epsilon \;\, \mbox{ and }\;\, \lim_{\mu \downarrow 0}\mbox{dist}\left(\nabla  \Upsilon^\lambda_\mu (z), \;\partial_C  \Upsilon^\lambda(z)\right) = 0.
\end{equation}
Note that in \eqref{Jacobian consistency}, \emph{dist} represents the standard distance function while $\partial_C  \Upsilon^\lambda(z)^T$ denotes the C-subdifferential
\begin{equation} \label{fischercdiff}
 \partial_C  \Upsilon^\lambda(z)^T := \partial  \Upsilon^\lambda_1 (z) \times ... \times \partial  \Upsilon^\lambda_{N+m}(z),
\end{equation}
commonly used in this context; see, e.g., \cite{kanz999}. In \eqref{fischercdiff}, $N:=n+m+2p+q$ and $\partial  \Upsilon^\lambda_i$, $i=1, \ldots, N+m$ represents the subdifferential in the sense of Clarke. Note that $\partial_C  \Upsilon^\lambda(z)^T$ contains the generalized Jacobian in the sense of Clarke of the function $\Upsilon^\lambda$.
Roughly speaking, the Jacobian consistency property \eqref{Jacobian consistency} translates to a framework ensuring that when the smoothing parameter $\mu$ is getting close to zero, the Jacobian $\nabla \Upsilon_\mu^\lambda (z)$ converges to an element in the C-subdifferential $\partial_C\Upsilon^\lambda(z)$. This property is important in determining the accuracy of the smoothing method in Algorithm \ref{algorithm 3s}.

Based on \eqref{fischercdiff}, at all points $z:=(x, y, u, v, w)$ satisfying Assumption \ref{AssumStrictComp},
\begin{equation} \label{partialCphi}
\partial_C  \Upsilon^\lambda(z)^T := \left\{\nabla\Upsilon^\lambda(z)^T\right\},
\end{equation}
where $\nabla \Upsilon^\lambda(z)$ is defined by \eqref{compactJac}.
For the case when strict complementarity does not hold, elements of $\partial_C \Upsilon^\lambda(z)$ have the same structure as in \eqref{compactJac} with the only differences being in the terms $\tau_j$, $\gamma_j$, $\alpha_j$, $\beta_j$, $\theta_j$, and $\kappa_j$ for indices $j$ where strict complementarity does not hold. We still have $\partial \Upsilon^\lambda_i (z)$ is the same as $\nabla \Upsilon^\lambda_i(z)$ for rows $i=1,...,n+2m$.
To determine the remaining rows, consider
\[
\begin{array}{rcl}
\Omega_1 &:=& \{j: (u_j, g_j (x,y)) = (0,0)\},\\
\Omega_2 &:=&\{j: (v_j, G_j (x,y)) = (0,0)\},\\
\Omega_3&:=& \{j: (w_j, g_j (x,y)) = (0,0)\}.
\end{array}
\]
Obviously $\tau_j$ and $\gamma_j$ introduced in \eqref{deftau} are not well-defined for $j\in\Omega_1$.
Similarly, the same holds for $\alpha_j$ and $\beta_j$ for $j\in \Omega_2$ and $\theta_j$ and $\kappa_j$ for $j\in\Omega_3$.
Following the same procedure as in \cite[Proposition 2.1]{kanz999}, we define
\[
\begin{array}{l}
\tau_k := \zeta_k + 1, \;\, \gamma_k := \rho_k - 1 \, \mbox{ for some }\, (\zeta_k,\rho_k) \in \mathbb{R}^2 \, \mbox{ such that }\, \norm{(\zeta_k, \rho_k)} \leq 1 \, \mbox{ if }\, k\in \Omega_1,\\
  \alpha_k:=\sigma_k+1,\;\, \beta_k:=\delta_k-1 \, \mbox{ for some }\, (\sigma_k,\delta_k) \in \mathbb{R}^2 \, \mbox{ such that }\, \norm{(\sigma_k, \delta_k)} \leq 1 \, \mbox{ if }\, k\in \Omega_2,\\
\theta_k:=\iota_k+1,\;\, \kappa_k:=\eta_k-1 \, \mbox{ for some }\, (\iota_k,\eta_k) \in \mathbb{R}^2 \, \mbox{ such that }\, \norm{(\iota_k, \eta_k)} \leq 1 \, \mbox{ if }\, k\in \Omega_3.
\end{array}
\]
Since we do not assume strict complementarity here, then in contrast to Subsection \ref{SubSecInvertibility}, we have
\[
(\tau_j-1)^2+(\gamma_j+1)^2 \leq 1, \;\; (\alpha_j-1)^2+(\beta_j+1)^2 \leq 1, \;\; (\theta_j-1)^2+(\kappa_j+1)^2 \leq 1.
\]

\begin{thm}
For $\lambda >0$, the Jacobian consistency property holds for the approximation $\Upsilon_\mu^\lambda$ of $\Upsilon^\lambda$.
\end{thm}
\begin{proof}
First of all, note that $\Upsilon^\lambda$ is locally Lipschitz continuous. Proceeding as
 in \cite[Corollary 2.4]{kanz999}, we can easily check that we have
\[
\norm{\Upsilon^\lambda_\mu(z) - \Upsilon^\lambda(z) } \leq  \epsilon \sqrt{\mu} \;\; \mbox{ with } \;\; \epsilon:= 2\sqrt{2p} + \sqrt{2q}.
\]
\begin{equation} \label{proveconsist}
\lim_{\mu \downarrow 0} \nabla \Upsilon^\lambda_{\mu } (z) =
\lim_{\mu \downarrow 0} \left[
  \begin{array}{cccc}
    \nabla^2 L^\lambda(z)  & \nabla g(x,y)^T & \nabla G(x,y)^T & -\lambda \nabla g(x,y)^T \\
     \nabla (\nabla_y\mathcal{L} (z)) & O & O & \nabla_y g(x,y)^T \\
    \mathcal{T}^\mu \nabla g (x,y) & \Gamma^\mu & O & O \\
     \mathcal{A}^\mu \nabla G (x,y) & O &  \mathcal{B}^\mu & O \\
    \Theta^\mu \nabla g (x,y) & O & O & \mathcal{K}^\mu \\
  \end{array}
\right],
\end{equation}
where it is easy to see that the first two rows of \eqref{proveconsist} are the same as the first two rows of \eqref{compactJac}, as these are not involving perturbation $\mu$. For the rest of Jacobian, we observe that
\begin{equation} \label{proveconsist2}
\lim_{\mu \downarrow 0} \left[ \tau_j^\mu \nabla g_j (x,y), \phantom{-} \gamma_j^\mu, \phantom{-} 0, \phantom{-} 0 \right]
=
\begin{cases}
 \left(\tau_j \nabla g_j (x,y) \phantom{-} \gamma_j \phantom{-} 0 \phantom{-} 0 \right)
 & \text{for } j \notin \Omega_1 \\
\left(\nabla g_j (x,y) \phantom{-} -1 \phantom{-} 0 \phantom{-} 0 \right) & \text{for } j\in \Omega_1
\end{cases}
\end{equation}
and similarly for $\lim_{\mu \downarrow 0} \big[ \big(\alpha_j^\mu \nabla G_j (x,y) \phantom{-} \beta_j^\mu \phantom{-} 0 \phantom{-} 0 \big) \big]$ and $\lim_{\mu \downarrow 0} \big[ \big(\theta_j^\mu \nabla g_j (x,y) \phantom{-} \kappa_j^\mu \phantom{-} 0 \phantom{-} 0 \big) \big]$. This leads to $\lim_{\mu \downarrow 0} dist (\nabla  \Upsilon^\lambda_\mu (z), \partial_C  \Upsilon^\lambda(z)) = 0$ as $\partial_C \Upsilon^\lambda(z)$ has the same form in \eqref{compactJac} with corresponding adjustments to $\zeta_j, \rho_j, \sigma_j, \delta_j, \iota_j$ and $\eta_j$ for the cases when strict complementarity does not hold.
\end{proof}

\section{Numerical results}\label{SecNumExperGaussNewt}

The focus of our experiments in this section will be on the smoothing system \eqref{bilevncp2222}, where we set $\mu:=10^{-11}$ constant throughout all iterations. Based on this system, we test and compare the Gauss-Newton method, Pseudo-Newton method, and the Matlab built-in method called \emph{fsolve}  (with Levenberg-Marquardt chosen as option). The examples used for the experiments are from the Bilevel Optimization LIBrary of Test Problems (BOLIB) \cite{bolib17}, which contains 124 nonlinear examples. The experiments are run in MATLAB, version R2016b, on MACI64. Here, we present a summary of the results obtained; more details for each example are reported in \cite{gnewtnumer18}.

For Step 0 of Algorithm \ref{algorithm 3s} and the corresponding smoothed Pseudo-Newton algorithm, we set the tolerance to $\epsilon:=10^{-5}$ (see Subsection \ref{Variation} for a justification) and the maximum number of iterations to be $K:=1000$. As for stopping criterion of $fsolve$, the tolerance is set to $10^{-5}$ as well.
For the numerical implementation we calculate the direction $d^k$ by solving \eqref{Gauss-Newtonsdireq} with Gaussian elimination.
Five different values of the penalty parameter are used for all the experiments; i.e., $\lambda \in \{100,10,1,0.1,0.01\}$, see \cite{gnewtnumer18} for details of the values of each solution for a selection of $\lambda$.
The motivation of using different values of $\lambda$ comes from the idea of not over-penalizing and not under-penalizing deviation of lower-level objective values from the minimum, as bigger (resp. smaller) values of $\lambda$ seem to perform better for small (resp. big) values of lower-level objective.
After running the experiments for all values of $\lambda\in \{100,10,1,0.1,0.01\}$, the best one is chosen (see Table 1 in \cite{gnewtnumer18}), i.e. for which the best feasible solution is produced for the particular problem by the tested algorithms.
Later in this section we present the comparison of the performance of the algorithms for the best value of  $\lambda$.
The experiments have shown that the algorithms perform much better if the starting point $(x^0,y^0)$ is feasible. As a default setup, we start with $x^0=1_n$ and $y^0=1_m$. If the default starting point does not satisfy at least one constraint, we choose a feasible starting point; see \cite{gnewtnumer18}. Subsequently, the Lagrange multipliers are initialised at $u^0 = \max\,\left\{0.01,\;-g(x^0,y^0)\right\}$,  $v^0=\max\,\left\{0.01,\;-G(x,y)\right\}$, and $u^0 = w^0$.

\subsection{Performance profiles}

Performance profiles are widely used to compare characteristics of different methods.
In this section we consider performance profiles, where $t_{p,s}$ denotes the CPU time to solve problem $p$  by algorithm $s$.
If the optimal solution of problem $p$ is known but it cannot be solved by algorithm $s$ (i.e., upper-level objective function error $>60\%$), we set $t_{p,s} := \infty$. We then define the \emph{performance ratio} by
$$
r_{p,s} := \frac{t_{p,s}}{\min \{ t_{p,s} : s \in S \}},
$$
where $S$ is the set of solvers.
Performance ratio is the ratio of how algorithm $s$ performed to solve problem $p$ compared to the performance of the best performed algorithm from the set $S$.
The \emph{performance profile} can be defined as the cumulative distribution function of the performance ratio:
$$\rho_s (\tau) := \frac{\big| \big\{ p\in P : r_{p,s} \leq \tau \big\} \big|}{n_p},$$
where $P$ is the set of problems.
Performance profile, $\rho_s (\tau)$, is counting the number of examples for which the performance ratio of the algorithm $s$ is better (smaller) than $\tau$.
The perfomance profile $\rho_s : \Re \rightarrow [0,1]$ is non-decreasing function, where the value of $\rho_s (1)$ shows the fraction of the problems for which solver $s$ was the best.

\begin{figure}[H]
\centering
\includegraphics[width=15cm]{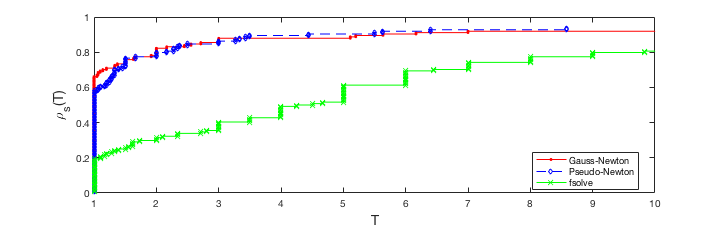}
\caption{Performance profiles of the methods for 124 problems}
\label{PerfProf}
\end{figure}

Comparing line-graphs of the performance profiles, the higher position of the graph means the better performance of the algorithms. The value on the y-axis shows the fraction of examples for which performance ratio is better than $T$ (presented on the x-axis).
As we set $t_{p,s} := \infty$ for the cases when we do not solve the problem, the algorithms would not go much over $90\%$ mark as for the rest of the problem $\rho_s(T) = \infty$.
Figure \ref{PerfProf} clearly shows that Gauss-Newton and Pseudo-Newton showed better performance than fsolve.
Since the variable for the comparison was CPU time, based on the values of $\rho_s(1)$, we can claim that Gauss-Newton was the fastest algorithm for about $70 \%$ of the problems, Pseudo-Newton for about $60\%$ of the problems and fsolve was the quickest for about $20 \%$ of the problems.
From the graph, one can also see that Gauss-Newton and Pseudo-Newton methods have
$\rho_s(2) = 80 \%$, while fsolve only has the value $\rho_s (2) = 30\%$, meaning that fsolve was more than twice worse than the best algorithm for $70\%$ of the problems.
Approaching $T=6$, Gauss-Newton and Pseudo-Newton getting performance ratio close to $\rho_s (T) = 90 \%$, where fsolve has the value $\rho_s (6) \approx 65\%$.
This clearly shows that Gauss-Newton and Pseudo-Newton methods show quite similar performance in terms of CPU time.
Both these algorithms clearly outperform fsolve for solving the set of test problems in terms of the measure of performance profiles.

\subsection{Feasibility check}
Considering the structure of the feasible set of problem \eqref{initialvalfuncform0}, it is critical to check whether the points computed by our algorithms satisfy the value function constraint $f(x,y)\leq \varphi(x)$. If the lower-level problem is convex in $y$ and a constraint qualification (e.g., the MFCQ) holds at a solution point generated by our algorithms, then this point will automatically satisfy the value function constraint, provided \eqref{kktbilev42} and \eqref{kktbilev62} hold. Note that the latter conditions are incorporated in the stopping criterion of Algorithm \ref{algorithm 3s}. To check whether the points obtained are feasible, we first identify the BOLIB examples, where the lower-level problem is convex w.r.t. $y$; see the summary of these checks in Table \ref{ConvTable}. It turns out that significant part of test examples have linear lower-level constraints. For these examples, a constraint qualification (CQ) is automatically satisfied. 
\begin{table}[H]
\begin{small}
\begin{center}
\begin{tabular}{l|l|l}
 \hline
 $f(\cdot,y)$ & $g_i(\cdot,y), i=1, \ldots, p$ & \textbf{Total count}  \\
  \hline
  Convex & Convex (linear) & 55 \\
  Convex & Convex (nonlinear) & 14 \\
  Convex & No constraints & 5 \\
  \hline
  Convex & Convex (with CQ satisfied) & 74 \\
  \hline
\end{tabular}
\end{center}
\caption{Convexity of the lower-level functions}\label{ConvTable}
\end{small}
\end{table}
There are 14 problems, where convex $f(\cdot,y)$ and $g_i(\cdot,y)$ with $i=1, \ldots, p$, are convex, but the constraints are not all linear w.r.t. $y$. For these examples the MFCQ has been shown to hold at the point computed by our algorithms. The rest of the problems have non-convex lower-level objective or some of the lower-level constraints being nonconvex. For these examples, we compare the obtained solutions  with the known ones from the literature.
Let $f_A$ stand for $f(\bar{x},\bar{y})$ obtained by one of the tested algorithms and $f_K$ to be the known optimal value of lower-level objective function.
In the graph below we have the lower-level relative error, $(f_A - f_K) / (1+|f_K|)$, on the y-axis.
The x-axis starts from $25th$ example and the error is plotted in increasing order.

\begin{figure}[H]
\centering
\includegraphics[scale=0.55]{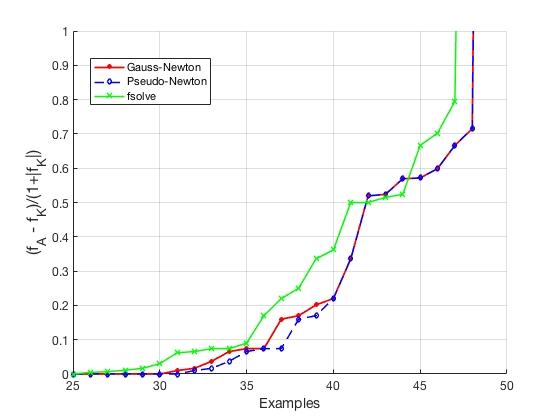}
\caption{Lower-level optimality check for examples with a nonconvex lower-level problem}
\label{Lowerlevhist}
\end{figure}
From the figure above we can see that for 30 problems the relative error of lower-level objective is negligible ($<5\%)$ for all three methods.
Almost for all of the remaining 19 examples Gauss-Newton and Pseudo-Newton have smaller errors than fsolve and Pseudo-Newton seems to have slightly smaller errors than Gauss-Newton method for some of the examples.
We have seen that convexity and a CQ hold for the lower-level hold for 74 test examples.
We consider solutions for these problems to be feasible for the lower-level problem.
Taking satisfying feasibility error to be $<20 \%$, we claim that feasibility is satisfied for 113 (91.13\%) problems for Gauss-Newton and Pseudo-Newton methods, and for 110 (88.71\%) problems for fsolve.

\subsection{Accuracy of the upper-level objective function}
Here, we compare the values of the upper-level objective functions at points computed by the algorithms; i.e., Gauss-Newton and
Pseudo-Newton algorithms, and fsolve. For this comparison purpose, we focus our attention only on  116 BOLIB examples \cite{bolib17}, as solutions are not known for six of them and the Gauss-Newton algorithm diverges for \emph{NieEtal2017e}, possibly due the singularity of the matrix $\nabla \Upsilon^\lambda_\mu (z)^T \nabla \Upsilon^\lambda_\mu (z)$; see \cite{gnewtnumer18} for more details. To proceed, let $\bar{F}_A$ be the value of upper-level objective function at the point $(\bar{x},\bar{y})$ obtained by one of the algorithms (Gauss-Newton, Pseudo-Newton or fsolve) and $\bar{F}_K$  the value of this function at the known best solution point reported in the literature (see corresponding references in \cite{bolib17}).
The comparison is shown in Figure \ref{Upperscaleshist}, where we have the relative error $(\bar{F}_A - \bar{F}_K) / (1+|\bar{F}_K|)$ on the $y-axis$ and number of examples on the x-axis, starting from $85th$ example. The graph is plotted in the order of increasing error.

\begin{figure}[H]
\centering
\includegraphics[scale=0.55]{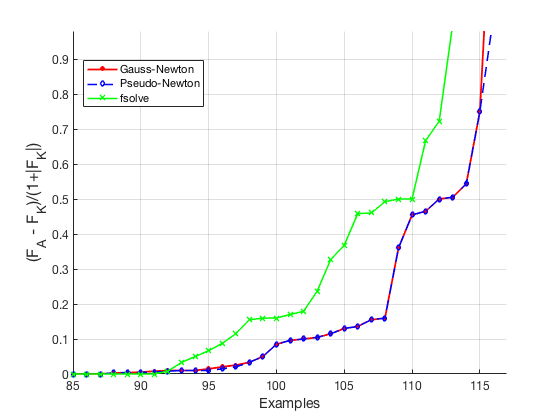}
\caption{Comparison of upper-level objective values for examples with known solutions}
\label{Upperscaleshist}
\end{figure}
From Figure \ref{Upperscaleshist}, we can see that most of the known best values of upper-level objective functions were recovered by all the methods, as the relative error is close to zero. Precisely, for 93 of the tested problems, the upper-level objective function error is negligible (i.e., less than $5 \%$) for the solutions obtained by all the three methods. For the remaining 23 examples, it is clear that the errors resulting from fsolve are much higher than the ones from the Gauss-Newton and Pseudo-Newton methods.
It is worth noting that algorithms perform fairly well for most of the problems.
With the accuracy error of $\leq20\%$ our algorithms recovered solutions for $92.31\%$ of the problems, while fsolve recovered only $88.03\%$ of the solutions.

\subsection{Variation of the tolerance in the stopping criterion}\label{Variation}
We are now going to evaluate the performance of Algorithm \ref{algorithm 3s} as we relax the tolerance in the stopping criterion.
Precisely, we set $\epsilon := 10^{-8}$ (as opposed to $\epsilon := 10^{-5}$ used so far) and it turns out that for most of the examples, this is not achievable. Hence, the algorithms then stop after the maximum number of iterations or if the gap between improvement from step to step is too small.


\begin{figure}[H]
\centering
\hspace*{-0.5cm}
\includegraphics[scale=0.7]{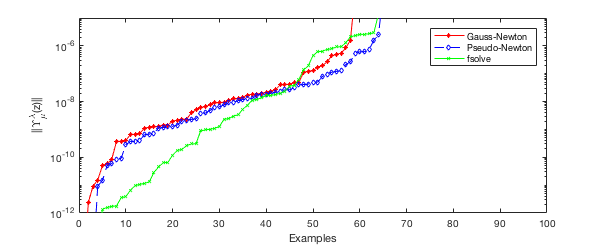}
\caption{Performance of the methods in terms of solving $\Upsilon^\lambda_\mu (z)=0$} \label{SystemToler}
\end{figure}
The values of $|| \Upsilon^\lambda_\mu (z)||$ produced by the algorithms are presented in increasing order on the y-axis in Figure \ref{SystemToler}. We can see that fsolve performs slightly better for 40 examples, where we have $||\Upsilon^\lambda_\mu (z)|| \leq 10^{-8}$. For those examples, fsolve recovered solution with better tolerance than Gauss-Newton and Pseudo-Newton algorithms.
This shows that whenever we are able to solve a problem almost exactly, fsolve's stopping criteria is more strict and obtains slightly better values of the system. This can be explained by an additional stopping criteria that we use for Gauss-Newton and Pseudo-Newton if the improvement between the steps of the algorithms gets too small.
This shows that whenever we are able to solve the system $\Upsilon^\lambda_\mu (z) = 0$
almost exactly, fsolve's stopping criteria is a bit less strict and iterations keep going to produce values of $||\Upsilon^\lambda_\mu (z)||$ that are closer to $0$ than for the other two methods. The explanation could be that due to an additional stopping criteria of Gauss-Newton and Pseudo-Newton 
the algorithms stop earlier once significant improvement of the solution is not observed from step to step.
If we now look at the performance of algorithms between tolerances of $10^{-8}$ and $10^{-6}$, the Pseudo-Newton method shows better performance than the other algorithms and fsolve is the weakest of the three algorithms. This means that if we want to solve a problem with the tolerance of $10^{-6}$ or better, the Pseudo-Newton algorithm is more likely to recover solutions than the other two methods.
The other important observation from Figure \ref{SystemToler} is that choosing $\epsilon := 10^{-5}$ is the most sensible tolerance as better tolerance only allow  about $50 \%$ of the examples to be solved (i.e., just over 60 examples as we can see from the graph).

\subsection{Checking assumption on $\lambda$}\label{SecLamassum}
Considering the importance of the requirement that  $\lambda < \frac{\kappa_j^\mu}{\theta_j^\mu} \frac{\tau_j^\mu}{\gamma_j^\mu}$ for $j=1, \ldots, p$ in Theorem \ref{GaussNewtLinIndeps}, let us analyse its behaviour at the solution point generated with Algorithm \ref{algorithm 3s} for each value of $\lambda \in \{100, \, 10, \, 1, \, 0.1, \, 0.01\}$. To simplify the analysis, we introduce
\[
c^\mu (\lambda) := \min_{j=1, \ldots, p}~\frac{\kappa_j^\mu}{\theta_j^\mu} \frac{\tau_j^\mu}{\gamma_j^\mu}.
\]
It suffices to check that $\lambda < c^\mu (\lambda)$.
We set $\lambda - c^\mu (\lambda) := 100$ if 
 the difference is undefined, while discarding problems where $g(x,y)$ is not present.
Note that the assumption on $\lambda$ is not necessary in Theorem \ref{GaussNewtLinIndeps} for problems with no lower-level constraint. Let us also introduce the notions of the best $\lambda$ and optimal $\lambda$, where
by the best $\lambda$ we mean the values of $\lambda$ for which $\lambda - c^\mu (\lambda)$ is the smallest and by the optimal we mean the values of $\lambda$ that were the best to obtain the solution according to \cite{gnewtnumer18}.
In the figure below we present the difference $\lambda - c^\mu (\lambda)$ on the y-axis and number of the example on the x-axis, following ascending order w.r.t. the values on the y-axis.

\begin{figure}[H]
\centering
\hspace*{-0.5cm}
\includegraphics[scale=0.6]{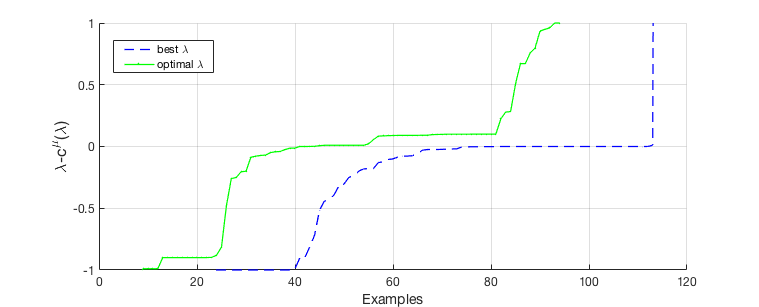}
\caption{Checking assumption $\lambda<c^\mu (\lambda)$ for best and optimal values of $\lambda$}
\label{lamassumfigure42}
\end{figure}
Clearly,  condition $\lambda < c^\mu (\lambda)$ holds for the values of $\lambda - c^\mu (\lambda)$ lying below the $x$-axis.
From Figure \ref{lamassumfigure42} we can see that the assumption holds for 42 (out of 116) problems for the optimal $\lambda$. This means that the condition can hold for many examples. But, obviously, as it is not a necessary condition, our Algorithm \ref{algorithm 3s} still converges for many other problems, where the condition is not necessarily satisfied.
For the best values of $\lambda$, condition $\lambda < c^\mu (\lambda)$ holds for 101 problems. Hence, showing that for most of the examples, there is at least one value of $\lambda \in \{100, \, 10, \, 1, \, 0.1, \, 0.01\}$ for which the condition is satisfied.

\subsection{Final comments}
In this paper, a class of the LLVF-based optimality conditions for bilevel optimization problems has been reformulated as a system of equations using the Fischer-Burmeister function.
It was shown that the Gauss-Newton method can be well-defined and the framework for convergence is provided. We have tested the method and its smoothed version numerically, alongside with Newton method with Moore-Penrose pseudo inverse. The comparison of the obtained solutions with known best ones showed that the methods are appropriate to be used for bilevel programs, recovering optimal solutions (when known) for most of the tested problems. It is worth mentioning that whenever  $\nabla \Upsilon^\lambda_\mu (z)^T \nabla\Upsilon^\lambda_\mu (z)$ is not ill-conditioned throughout all iterations, Gauss-Newton and Pseudo-Newton methods produced the same results as expected.
More interestingly, for the 38 problems for which Gauss-Newton could not be implemented due to singularity of the direction matrix for one or more values of $\lambda$ (see \cite{gnewtnumer18}), our conjecture that Pseudo-Newton would produce reasonable solutions for these cases worked well for 14 problems (e.g. \emph{'CalamaiVicente1994c', 'DempeDutta2012b', 'DempeFranke2011a'} in \cite{gnewtnumer18}) and failed for the remaining 24 examples (e.g.  \emph{'Bard1988c', 'Colson2002BIPA3', 'DempeDutta2012a'} in \cite{gnewtnumer18}). Nevertheless, we can say that the Pseudo-Newton method is indeed slightly more robust compared to the Gauss-Newton method. This, together with slightly better feasibility (seen  in Figure \ref{Lowerlevhist}) and very similar computation time, allows us to say that Pseudo-Newton method is indeed slightly more robust than Gauss-Newton method.

All the three methods (Gauss-Newton, Pseudo-Newton, and fsolve) show very fast performance with average CPU time of less than $0.5$ seconds for any of the five values of $\lambda$ used. However, the Gauss-Newton and Pseudo-Newton methods are more efficient at recovering solutions for the 124 BOLIB test problems from \cite{bolib17}. Apart from recovering more solutions they also showed better CPU time than fsolve as seen from performance profiles. Our methods also showed to produce more feasible solutions than fsolve in the sense of the feasibility for the lower-level problem.

\end{document}

%% file: structure.tex
\usepackage[
nochapters, 
beramono, 
eulermath,
pdfspacing, 
dottedtoc 
]{classicthesis} 

\usepackage{arsclassica} 

\usepackage[T1]{fontenc} 

\usepackage[utf8]{inputenc} 

\usepackage{graphicx} 
\graphicspath{{Figures/}} 

\usepackage{enumitem} 

\usepackage{lipsum} 

\usepackage{subfig} 

\usepackage{amsmath,amssymb,amsthm} 

\usepackage{varioref} 

\theoremstyle{definition} 

\theoremstyle{plain} 

\theoremstyle{remark} 

\hypersetup{
colorlinks=true, breaklinks=true, bookmarks=true,bookmarksnumbered,
urlcolor=webbrown, linkcolor=RoyalBlue, citecolor=webgreen, 
pdftitle={}, 
pdfauthor={\textcopyright}, 
pdfsubject={}, 
pdfkeywords={}, 
pdfcreator={pdfLaTeX}, 
pdfproducer={LaTeX with hyperref and ClassicThesis} 
}

%% file: FliegeTinZemkoho.bbl
\begin{thebibliography}{99}
\bibitem{Allendestill12}
G.B. Allende  and G. Still.
\newblock {\it Solving bi-level programs with the KKT-approach},
\newblock Mathematical Programming 131:37-48 (2012)


\bibitem{globandsuperlin98}
X. Chen, L. Qi, and D. Sun.
\newblock {\it Global and superlinear convergence of the smoothing Newton method and its application to general box constrained variational inequalities}, \newblock Mathematics of Computation 67(222):519-540 (1998)

\bibitem{dempezemkoho1}
S. Dempe  and A.B. Zemkoho.
\newblock {\it The generalized Mangasarian-Fromowitz constraint qualification and optimality conditions for bilevel programs},
\newblock Journal of Optimization Theory and Applications 148(1):46-68 (2011)

\bibitem{newoptcond}
S. Dempe, J. Dutta, and B.S. Mordukhovich.
\newblock {\it New necessary optimality conditions in optimistic bilevel programming},
\newblock Optimization 56 (5-6):577-604
(2007)

\bibitem{bilevelreform}
S. Dempe  and A.B. Zemkoho.
\newblock {\it The bilevel programming problem: reformulations, constraint qualification and optimality conditions},
\newblock Mathematical Programming 138:447-473
(2013)

\bibitem{bilevelmpec10}
S. Dempe  and J. Dutta.
\newblock {\it Is bilevel programming a special case of mathematical programming with equilibrium constraints?}
\newblock Mathematical Programming 131:37-48
(2010)

\bibitem{dennisbook96} 
J.E. Dennis  and R.B. Schnabel.
\newblock {Numerical methods for unconstrained optimization and nonlinear equations},
\newblock SIAM Classics in Applied Mathematics, 1996

\bibitem{fischer92}
A. Fischer.
\newblock {\it A special Newton-type optimization method},
\newblock Optimization 24(3):269-284 (1992)

\bibitem{newtonbilevel18}
A. Fischer, A.B. Zemkoho, and S. Zhou.
\newblock {\it Semismooth Newton-type method for bilevel optimization: global convergence and extensive numerical experiments},
\newblock arXiv, {arXiv:1912.07079} (2019)

\bibitem{fletch}
R. Fletcher.
\newblock  Practical methods of optimization (2nd Ed.),
\newblock  John Wiley, 1987

\bibitem{fletcher68}
R. Fletcher.
\newblock {\it Generalized inverse methods for the best least squares solution of systems of non-linear equations},
\newblock The Computer Journal 10(4):392-399
(1968)

\bibitem{gatilov14}
S.Y. Gatilov.
\newblock {\it Using low-rank approximation of the Jacobian matrix in the Newton-Raphson method to solve certain singular equations},
\newblock Journal of Computational and Applied Mathematics 272:8-24
(2014)

\bibitem{golub96}
G.H. Golub  and C.F. Van Loan.
\newblock {Matrix computations},
\newblock The John Hopkins University Press, 1996


\bibitem{Yanchong13}
Y. Jiang, X. Li, C. Huang, and X. Wu.
\newblock {\it W. Application of particle swarm optimization based on CHKS smoothing function for solving nonlinear bi-level programming problem},
\newblock Applied Mathematics and Computation 219:4332-4339
(2013)

\bibitem{kanzow1996}
C. Kanzow.
\newblock {\it Some noninterior continuation methods for linear complementarity problems},
\newblock SIAM Journal on Matrix Analysis and Applications 17(4):851-868
(1996)

\bibitem{kanz999}
C. Kanzow  and H. Pieper.
\newblock {\it Jacobian smoothing methods for general nonlinear complementarity problems}, SIAM Journal on Optimization 9:342-372
(1999)

\bibitem{polyxeni141}
P. Kleniati  and C.S. Adjiman,
\newblock {\it Branch-and-sandwich: a deterministic global optimization algorithm for optimistic bilevel programming problems. Part I: Theoretical development},
\newblock Journal of Global Optimization 60(3):425-458 (2014)

\bibitem{polyxeni142}
P. Kleniati  and C.S. Adjiman.
\newblock {\it Branch-and-sandwich: a deterministic global optimization algorithm for optimistic bilevel programming problems. Part II: Convergence analysis and numerical results},
\newblock Journal of Global Optimization 60(3): 459-481 (2014)

\bibitem{lin14}
G.-H. Lin, M. Xu, and J.J. Ye.
\newblock {\it On solving simple bilevel programs with a nonconvex lower level program},
\newblock Mathematical Programming 144(1-2):277-305 (2014)

\bibitem{mitsos08}
A. Mitsos, P. Lemonidis, and P.I. Barton.
\newblock {\it Global solution of bilevel programs with a nonconvex inner program},
\newblock Journal of Global Optimization 42(4):475-513 (2008)

\bibitem{nocedalw}
J. Nocedal and S.J. Wright.
 Numerical optimization, Springer, 1999

\bibitem{pan91}
V. Pan  and R. Schreiber.
\newblock {\it An improved newton iteration for the generalized inverse of a matrix, with applications},
\newblock SIAM Journal on Scientific and Statistical Computing 12(5):1109-1130
(1991)

\bibitem{paulavicius17}
R. Paulavicius, J. Gao, P. Kleniati, and C.S. Adjiman.
\newblock {\it BASBL: Branch-and-sandwich bilevel solver. Implementation and computational study with the BASBLib test sets},
Computers \& Chemical Engineering 132:106609 (2020)

\bibitem{SunYuan06}
W. Sun and Y.-X. Yuan.
\newblock {Optimization Theory and Methods},
Springer, 2006


\bibitem{QiSunAsurvey}
L. Qi  and D. Sun,
\newblock {\it A survey of some nonsmooth equations and smoothing Newton methods},
\newblock In Progress in optimization Vol. 30 (pp. 121-146), Springer, 1999



\bibitem{sun_ncp}
D. Sun  and L. Qi.
\newblock {\it On NCP Functions},
\newblock Computational Optimization and Applications 13(1-3):201-220 (1999)

\bibitem{gnewtnumer18}
J. Fliege, A. Tin, and A.B. Zemkoho.
\newblock {\it Supplementary material for ``Gauss-Newton-type methods for bilevel optimization''},
\newblock School of Mathematical Sciences, University of Southampton, UK (2020)


\bibitem{wieseman13}
W. Wiesemann, A. Tsoukalas, P. Kleniati, and B. Rustem.
\newblock {\it Pessimistic bilevel optimization},
\newblock SIAM Journal on Optimization 23(1):353-380
(2013)

\bibitem{xu14}
M. Xu  and J.J. Ye.
\newblock {\it A smoothing augmented lagrangian method for solving simple bilevel programs},
\newblock Computational Optimization and Applications 59(1-2):353-377 (2014)

\bibitem{xu15}
M. Xu, J.J. Ye, and L. Zhang.
\newblock {\it Smoothing sqp methods for solving degenerate nonsmooth constrained optimization problems with applications to bilevel programs},
\newblock SIAM Journal on Optimization 25(3):1388-1410
(2015)

\bibitem{optcondbil95}
J.J. Ye  and D.L. Zhu.
\newblock {\it Optimality conditions for bilevel programming problems},
\newblock  Optimization 33:9-27
(1995)

\bibitem{zemkohothesis}
A.B. Zemkoho.
\newblock {\it Bilevel programming: reformulations, regularity and stationarity},
\newblock PhD thesis, Department of Mathematics and Computer Science, TU Bergakademie Freiberg, Freiberg, Germany
(2012)


\bibitem{bolib17}
S. Zhou, A.B. Zemkoho, and A. Tin.
\newblock {\it BOLIB: Bilevel Optimization LIBrary of Test Problems}, available at\\
\newblock \href{https://biopt.github.io/bolib/}{biopt.github.io/bolib} 
(2018)
\end{thebibliography}
